\documentclass[reqno,12pt]{amsart}
\usepackage{amsmath,amssymb,epsfig,color,here,wasysym,stmaryrd,slashbox,tikz-cd,url}

\numberwithin{equation}{section}

\newtheorem{theorem}{Theorem}[section]

\newtheorem{lemma}[theorem]{Lemma}
\newtheorem{proposition}[theorem]{Proposition}
\newtheorem{definition}[theorem]{Definition}
\newtheorem{corollary}[theorem]{Corollary}

\theoremstyle{remark}
\newtheorem{example}[theorem]{Example}
\newtheorem{remark}[theorem]{Remark}

\newcommand{\vanish}[1]{}

\begin{document}

\title{Spanning hypertrees, vertex tours and meanders}

\author[Robert Cori and G\'abor Hetyei]{Robert Cori \and G\'abor Hetyei}

\address{Labri, Universit\'e Bordeaux 1, 33405 Talence Cedex, France.
\hfill\break
WWW: \tt http://www.labri.fr/perso/cori/.}

\address{Department of Mathematics and Statistics,
  UNC Charlotte, Charlotte NC 28223-0001.
WWW: \tt http://webpages.uncc.edu/ghetyei/.}

\date{\today}


\begin{abstract}
  This paper revisits the notion of a spanning hypertree of a hypermap
  introduced by one of its authors and shows that it allows to shed new 
  light on a very diverse set of recent results.

  The tour of a map along one of its spanning trees used by
  Bernardi may be generalized to hypermaps and we show that it is
  equivalent to a dual tour described by Cori~\cite{Cori-hrec} and
  Mach\`\i~\cite{Machi}. We give a bijection between the spanning
  hypertrees of the reciprocal of the plane graph with $2$ vertices and $n$
  parallel edges and the meanders of order $n$ and a bijection of the same kind
  between semimeanders of order $n$ and spanning hypertrees of the
  reciprocal of a plane graph with a single vertex and $n/2$ nested
  edges. We introduce hyperdeletions and hypercontractions in a hypermap
  which allow to count the spanning hypertrees
  of a hypermap recursively, and create a link with the computation of
  the Tutte polynomial of a graph. Having a particular interest in
  hypermaps which are reciprocals of maps, we generalize
  the reduction map introduced by Franz and Earnshaw to enumerate
  meanders to a reduction map that allows the
  enumeration of the spanning hypertrees of such hypermaps. 
\end{abstract}

\dedicatory{\`A la m\'emoire de Pierre Rosenstiehl}

\maketitle

\section*{Introduction}

This paper is about {\em hypermaps}, a notion that has interested several
researchers in combinatorics. This notion generalizes that of a
{\em combinatorial map} (sometimes also called a {\em ribbon graph}), which
represents the embedding of a graph into an orientable 
surface with a pair of a permutation and of a fixed point free
involution. The same way hypergraphs generalize graphs by introducing hyperedges
incident to more than two vertices, hypermaps generalize maps by
replacing the involution with a permutation that has cycles of arbitrary
length. Hence one may interpret a hypermap as an embedding of a
hypergraph into an orientable surface. The main goal of this paper is to
return to the notion of a {\em spanning hypertree} of a hypermap introduced by
Cori, Penaud~\cite{Cori-hrec,Cori-Penaud} and Mach\`\i~\cite{Machi} with the
purpose of showing that several recent 
results on various combinatorial objects may be enlightened by
interpreting them in terms of spanning hypertrees of certain families of
hypermaps. These results concern apparently very distant areas such as
the tour of a graph with the purpose of computing its Tutte
polynomial~\cite{Bernardi-embeddings} or the determination of the number
of meanders and semimeanders. The main results of this paper are the following:
\begin{itemize}
\item Theorem~\ref{thm:0trees} and its Corollary~\ref{cor:vertextour}
  which generalizes the ``motion function'' used by
  Bernardi~\cite{Bernardi-embeddings} to
  hypermaps and provides a simple treatment in this more general setting.
\item Theorems~\ref{thm:mut} and~\ref{thm:mut0} characterizing the
  process of deletions and contractions  in a hypermap which allow to
  obtain all spanning hypertrees of a hypermap. These use a result of
  Goulden and Young~\cite[Theorem
  2.2]{Goulden-Yong} on a minimal decomposition of a permutation into a
  product of transpositions.
\item Theorem~\ref{thm:stdecomp} which gives a formula counting the number of
    spanning hypertrees of a hypermap in terms of the spanning
    hypertrees of a set of hypermaps obtained by deletions of
    contractions from the original hypermap.
\item Theorems~\ref{thm:semimeanders} and \ref{thm:meanders} which
  provide bijections between the spanning hypertrees of the reciprocal
  hypermaps of a plane graph with $2$ vertices and $n$ edges and
  meanders of order $n$, respectively bijections of the same type
  between semimeanders and and plane maps with one vertex and $n/2$
  parallel loops.  
\item Theorem~\ref{thm:rtrees} and its Corollary~\ref{cor:rtrees}  which
  generalize a result of Franz on meanders to hypermaps that are
  reciprocals of maps. 
\end{itemize}
Our paper is divided into seven sections, discussing the above mentioned
questions. 

In the preliminary Section~\ref{sec:prelim} we remind the reader of the
definition of a hypermap as a pair of permutations, one representing the
vertices the other one the hyperedges, the faces 
may then be expressed by a composition of these permutations and one may
define the genus by counting cycles. We also define some simple
transforms of hypermaps, such 
as the reciprocal (obtained by exchanging the vertices and the
hyperedges) and the dual (obtained by exchanging the vertices and the
faces). The 
central notion of a spanning hypertree is then introduced. It
relies on an order on the permutations based on their cycle
decompositions. Section~\ref{sec:Machi} recalls a result obtained by
Mach\`\i~\cite{Machi}, generalizing the result of Cori~\cite{Cori-hrec} and
Cori and Penaud~\cite{Cori-Penaud} showing the connection between the spanning
hypertrees and various other parameters of the pairs of permutations.
In Section~\ref{sec:hyperdc} we introduce the hyperdeletion and
hypercontraction operations for hypermaps, each of these multiply the
constituting permutations by a single transposition. These operations
generalize the well-known deletion and contraction operations on graphs
and combinatorial maps. Finding a spanning hypertree amounts to a
sequence of operations based on writing permutations as products of
transpositions. Using a dual description of face tours and vertex tours, 
we show that every hypermap has a {\em two-disk diagram} where vertices
form a noncrossing partition inside a vertex tour, faces form another
noncrossing partition inside the face tour. Drawing a diagram of a
hypermap thus corresponds to drawing a bipole (a hypermap with two
vertices) on a surface of the same genus as that of a hypermap, and then
adding some detail in a noncrossing fashion inside the two vertices,
representing the face tour and the vertex tour, respectively.
The sequence of operations introduced in Section~\ref{sec:hyperdc} is
described in detail in Section~\ref{sec:dcp} which opens a pathway to
the construction of a 
Tutte polynomial. Based on the results in Section~\ref{sec:dcp}
there would be several ways to define a Tutte polynomial. We did not
commit to any specific choice, because any such definition would depend
on the ordering of the hyperedges, unfortunately. It should be noted that
there is a hypergraph Tutte polynomial defined by Bernardi, K\'alm\'an and
Postnikov~\cite[Formula (2)]{Bernardi-Kalman-Postnikov} which is
independent of the ordering of the hyperedges and relies on labeling bases in a
(poly)matroid (see also \cite{Kalman, Kalman-Postnikov, Kalman-Tothmeresz} for related constructions). It is unclear, however, whether this approach could be
extended to hypermaps, and it should be noted that in the
Bernardi-K\'alm\'an-Postnikov construction some elements in the
ground set need to be considered simultaneously externally and
internally active, which hints at more complex relation with
deletion-contraction processes. Section~\ref{sec:shtc} 
provides a recursive formula for the number of 
spanning hypertrees in terms of certain hypermaps obtained by
hyperdeletions and hypercontractions. Our formula uses the same
decomposition of the noncrossing partition lattice as the one used by  
Simion and Ullman~\cite[Theorem 2]{Simion-Ullman} as an aid to
recursively construct a symmetric chain decomposition of the noncrossing
partition lattice. Section~\ref{sec:semimeanders}  is dedicated to
semimeanders and meanders, whose enumeration is the interest of many
authors, among whom Rosenstiehl was a first~\cite{Rosenstiehl}. We show
that their enumeration may be reduced to counting the 
spanning hypertrees of particular hypermaps: reciprocals of monopolar,
respectively bipolar maps with noncrossing parallel edges. This study,
together with the observation that besides duality taking the reciprocal
of a hypermap is part of the hypermap analogue of Tutte's trinity, makes
one think that reciprocals of maps must have special properties. This leads to
Section~\ref{sec:rtrees} where we generalize the work of
Franz~\cite{Franz-po,Franz-representation}, and we develop 
a labeled tree representation of the spanning hypertrees of the
reciprocal of a map. In particular, we show that the set of these
spanning hypertrees is bijectively equivalent to all trees that can be
obtained from the map by a sequence of topological vertex splittings.
This observation allows us to generalize the idea of Franz
and Earnshaw~\cite{Franz-Earnshaw} of a constructive enumeration of
meanders to an idea of a constructive enumeration of all spanning
hypertrees of the reciprocal of a map. 

\section{Preliminaries}
\label{sec:prelim}

\subsection{Hypermaps and hypertrees}

Informally, a {\em hypermap} is a hypergraph, topologically embedded in
a surface. Formally, it is a pair of permutations $(\sigma,\alpha)$
acting on the same finite set of labels, generating a transitive permutation
group. Fig.~\ref{fig:hypermap} represents the planar
hypermap $(\sigma,\alpha)$ for $\sigma=(1,2,3)(4,5,6)(7,8,9,10)(11,12)$ and
$\alpha=(1,6)(2,11,9,5)(3,7)(4,10)(8,12)$. The cycles of $\sigma$ are the {\em
  vertices} of the hypermap, the cycles of $\alpha$ are its {\em
  hyperedges}. A hypermap is a {\em map} if the 
length of each cycle in $\alpha$ is at most $2$.

\begin{figure}[h]
\begin{center}
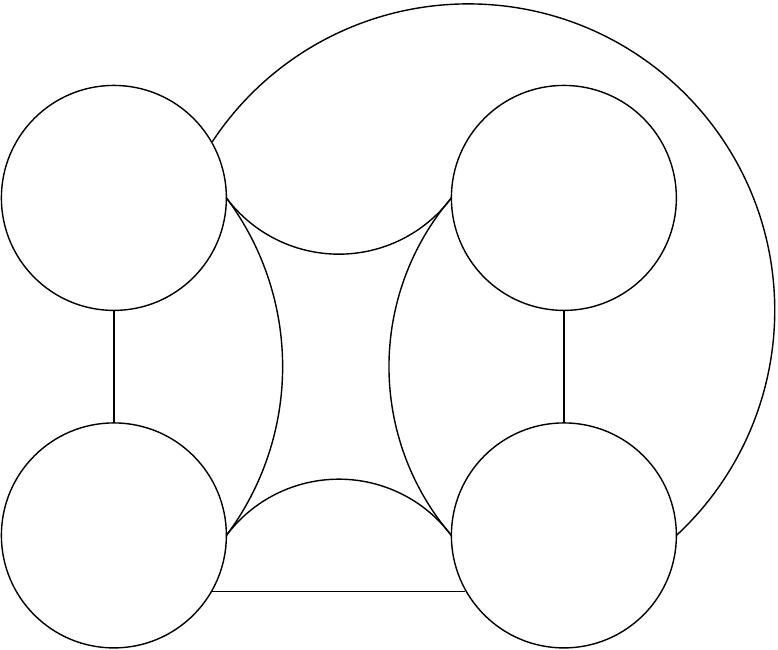
\end{center}
\caption{The hypermap $(\sigma,\alpha)$}
\label{fig:hypermap} 
\end{figure}

For planar hypermaps it is convenient to choose some
drawing conventions. In Fig.~\ref{fig:hypermap} we follow
the following rules: the cycles of $\sigma$ list the exits in
counterclockwise order, and we place the labels on the left hand side of each
exit (seen from the vertex). Following these conventions, the cycles of
$\alpha^{-1}\sigma=(1,5)(2,7,12)(3,6,10)(4,9)(8,11)$ label the regions
in the plane, created by the vertices and hyperedges, and we call the
cycles of the 
permutation $\alpha^{-1}\sigma$ the {\em faces} of the hypermap. Note
that {\em we multiply permutations from the right to the left}, in other
words we compose them as functions. One of our main sources,
Mach\`\i~\cite{Machi}, multiplies permutations from the left to the
right, and when he defines the faces as the cycles of
$\alpha^{-1}\sigma$, these are the cycles of $\sigma\alpha^{-1}$ in our
notation. When following Mach\`\i 's convention, it is more convenient to place
the labels on the right of each exit, thus the cycles of
$\sigma\alpha^{-1}$ label the regions created by the vertices and
hyperedges. 

There is a well-known formula, due to Jacques~\cite{Jacques} determining
the smallest genus $g(\sigma,\alpha)$ of a surface on which a hypermap
$(\sigma,\alpha)$ may be drawn. This number is given by the equation 
\begin{equation}
\label{eq:genusdef}
n + 2 -2g(\sigma,\alpha) = z(\sigma) + z(\alpha) + z(\alpha^{-1}
\sigma),
\end{equation}
where $z(\pi)$ denotes the number of cycles of the permutation $\pi$ and
$n$ is the number of permuted elements.
The number $g(\sigma,\alpha)$ is always an integer and it is called {\em
  the genus of the hypermap $(\sigma,\alpha)$}. In our example, $n=12$,
$z(\sigma)=4$, $z(\alpha)=5$, $z(\alpha^{-1}\sigma)=5$ and
Equation~\eqref{eq:genusdef} gives $g(\sigma,\alpha)=0$, that is, we
have a planar hypermap.  

A key notion of this paper is the {\em spanning hypertree of a
  hypermap}, first introduced for planar hypermaps in~\cite{Cori-hrec}
and generalized to hypermaps of arbitrary genus in~\cite{Machi}. A
hypermap $(\sigma,\alpha)$ is {\em a unicellular hypermap} if it
has only one face. We call a unicellular hypermap a {\em hypertree} if its
genus is zero. Note that Mach\` i~\cite{Machi} uses the term hypertree
even for unicellular hypermaps
of a higher genus. At this point our terminology is in line
with the widely used term {\em unicellular map} which is a map with only
one face. A permutation
$\theta$ is a {\em refinement} of a permutation $\gamma$, if there exists a
pair of decompositions $\gamma=\gamma_1\cdots \gamma_t$  and
$\theta=\theta_1\cdots \theta_t$ such that the $\gamma_i$s are
pairwise disjoint cycles of $\gamma$, the $\theta_i$s are products of
disjoint cycles of $\theta$, for each $i$ the permutations
$\gamma_i$ and $\theta_i$ act on the same set of elements, and they satisfy
$g(\theta_i,\gamma_i)=0$. It has been shown in~\cite[Theorem~1]{Cori}
that for a circular permutation $\sigma$ (that is, a permutation with a
single cycle) and an arbitrary permutation $\alpha$, acting on the same
set of elements, the condition $g(\sigma,\alpha)=0$ is equivalent to
requiring that the cycles of $\alpha$ list the elements of a {\em
  noncrossing partition}  according to the circular order determined by
$\sigma$. Hence the definition of a refinement may be equivalently
restated by requiring that, starting with the decomposition
$\gamma=\gamma_1\cdots \gamma_t$ of $\gamma$ into pairwise disjoint
cycles, we replace each cycle with a permutation $\theta_i$ whose cycles
represent a noncrossing partition of the points moved by $\gamma_i$,
with respect to the circular order represented by $\gamma_i$. 
The following is an immediate consequence of
the definition of a refinement.
\begin{corollary}
\label{cor:refine2}  
Let $\gamma$ and $\theta$ be permutations of the same set. Then $\theta$
is a refinement of $\gamma$ if and only if $\theta^{-1}$
is a refinement of $\gamma^{-1}$.
\end{corollary}  
Furthermore, refinements can be characterized as follows.
\begin{proposition}
\label{prop:refinement}
Let $\theta$ and $\gamma$ be permutations of an $n$-element set. 
The permutation $\theta$ is a refinement of $\gamma$ if and only if the
two following conditions are satisfied:
\begin{enumerate}
\item For any i, the two elements i and $\theta(i)$ are in the same cycle of
  $\gamma$.
\item $z(\theta^{-1}\gamma) + z(\theta) = n + z(\gamma)$ holds.   
\end{enumerate}  
\end{proposition}
\begin{proof}
The first condition is equivalent to stating that each for each cycle of
$\gamma$ there is a product of cycles of $\theta$ acting on the
same set. Introducing $n_i$ as the cycle length of $\gamma_i$, the
condition $g(\theta_i,\gamma_i)=0$ is equivalent to
$$
0=n_i+2-z(\gamma_i)-z(\theta_i)-z(\gamma_i^{-1}\theta_i). 
$$
The second condition may be obtained by adding all equations of the
above form, keeping in mind that the number of equations to be added is
$z(\gamma)$. Conversely, the second condition implies that all
$g(\theta_i,\gamma_i)=0$ as the genus of any hypermap is nonnegative.  
\end{proof}  
\begin{corollary}
\label{cor:refine1}  
If $\gamma$ and $\theta$ are permutations of the
same $n$-element set then $\theta$ is a refinement of $\gamma$ if and
only if the same holds for $\theta^{-1}\gamma$.
\end{corollary}
Notice that since the
pairs $\gamma_i, \theta_i$ define {\em monopoles} (that is, hypermaps
with a single vertex) of genus $0$, for a given
$\gamma_i$ the number of $\theta_i$ which are a refinements of
$\gamma_i$ is the Catalan number $C_{n_i}$, where $n_i$ is the number
of elements of the cycle $\gamma_i$. Thus we obtain the following. 
\begin{proposition}
Let $\gamma$ be a permutation whose cycles are of length $n_1,n_2,
\ldots, n_k$, respectively. Then the number of refinements of $\gamma$ is
$\prod_{i=1}^k C_{n_i}$. 
\end{proposition}

A hypermap $(\sigma,\alpha')$ {\em spans}
the hypermap $(\sigma,\alpha)$ if $\alpha'$ is a refinement of
$\alpha$. Note that not all refinements $\alpha'$ of $\alpha$ have the
property that $(\sigma,\alpha')$ is a hypermap. The {\em complexity} of
a hypermap of genus $g$ is the number of unicellular hypermaps of genus
$g$ spanning it.

Using Proposition~\ref{prop:refinement} we may establish a bijection between the
spanning genus $k$ unicellular hypermaps of a hypermap $(\sigma,\alpha)$
and the spanning genus $(g(\sigma,\alpha)-k)$ unicellular hypermaps of
its {\em dual} $(\alpha^{-1}\sigma, \alpha^{-1})$ for
$k=0,1,\ldots,g(\sigma,\alpha)$ as follows. 

\begin{theorem}
\label{thm:dhtree}  
Let $(\sigma, \alpha)$ be a hypermap and let $\theta$ be a permutation
of the same set of points. Then $(\sigma, \theta)$ is a spanning
unicellular hypermap of $(\sigma, \alpha)$ if and only if 
$(\alpha^{-1}\sigma,\alpha^{-1}\theta)$  is a 
spanning unicellular hypermap of the dual hypermap
$(\alpha^{-1}\sigma,\alpha^{-1})$. Furthermore, if the above are
satisfied  we have
$$
g(\sigma,\theta) + g(\alpha^{-1}\sigma,\alpha^{-1}\theta) = g(\sigma,\alpha).
$$
\end{theorem}  
\begin{proof}
By Corollary~\ref{cor:refine1}, $\theta$ is a refinement of $\alpha$ if
and only if $\theta^{-1}\alpha$ is a refinement of $\alpha$. By
Corollary~\ref{cor:refine2}, $\theta^{-1}\alpha$ is a refinement of
$\alpha$ if and only if $\alpha^{-1}\theta$ is a refinement of
$\alpha^{-1}$. Observe furthermore that the face permutation of
$(\alpha^{-1}\sigma,\alpha^{-1}\theta)$ is
$$
(\alpha^{-1}\theta)^{-1}\alpha^{-1}\sigma=
\theta^{-1}\alpha\alpha^{-1}\sigma=
\theta^{-1}\sigma.
$$
Combining the above observations we obtain that 
$(\sigma, \theta)$ is a spanning
unicellular hypermap of $(\sigma, \alpha)$ if and only if 
$(\alpha^{-1}\sigma,\alpha^{-1}\theta)$  is a 
spanning unicellular hypermap of $(\alpha^{-1}\sigma,\alpha^{-1})$.

Finally, the stated equation connecting the genuses holds because
of~(\ref{eq:genusdef}) and the second statement in
Proposition~\ref{prop:refinement}. Indeed, (\ref{eq:genusdef}) yields
\begin{align*}
  g(\sigma,\theta)&=
  \frac{1}{2}\cdot \left(n+2-z(\sigma)-z(\theta)-z(\theta^{-1}\sigma)\right)\\ 
  g(\alpha^{-1}\sigma, \alpha^{-1}\theta) &=
  \frac{1}{2}\cdot
  \left(n+2-z(\alpha^{-1}\sigma)-z(\alpha^{-1}\theta)-z(\theta^{-1}\sigma)\right) 
  \text{ and}\\   
  -g(\sigma,\alpha)&=
  \frac{1}{2}\cdot \left(-n-2+z(\sigma)+z(\alpha)+z(\alpha^{-1}\sigma)\right).\\\end{align*}  
  Using $z(\theta^{-1}\sigma)=1$ we obtain
\begin{align*}
g(\sigma,\theta) + g(\sigma\alpha^{-1}, \theta\alpha^{-1})-
g(\sigma,\alpha)&=
\frac{1}{2}\cdot\left(n-z(\theta)-z(\alpha^{-1}\theta)+z(\alpha)\right).
\end{align*}
We may replace $z(\alpha^{-1}\theta)$ with 
$z((\alpha^{-1}\theta)^{-1})=z(\theta^{-1}\alpha)$
on the right hand side of the last equation, which is then is zero by
the second statement in Proposition~\ref{prop:refinement}.
\end{proof}

Note that the drawing conventions stated for planar hypermaps above may be
easily generalized to hypermaps drawn on an oriented surface of a fixed
genus. As noted earlier, in such a setting, the definition of a
refinement $\theta$ of $\gamma$ requires to replace each cycle
$\gamma_i$ of $\gamma$ by a noncrossing partition in which the parts respect the cyclic order of $\gamma_i$.
Furthermore, a permutation $\theta'$ refines the permutation $\theta$
further exactly when each cycle of $\theta'$ is contained in a cycle
$\gamma_i$ of $\gamma$, and the restriction of $\theta'$ onto the set of
elements permuted by $\gamma_i$ is a noncrossing partition, which is a
refinement of the noncrossing partition associated to the action of
$\theta$ on the elements permuted by $\gamma_i$. For noncrossing
partitions we use the term refinement in the same sense as
Kreweras~\cite{Kreweras}, who has shown that noncrossing partitions form
a lattice under refinement.

The key result we use depends on the notion of the {\em hyperdual} of a
hypermap, introduced in~\cite{Cori-Penaud}. The {\em dual} of the
hypermap $(\sigma,\alpha)$ is the hypermap
$(\alpha^{-1}\sigma,\alpha^{-1})$, this notion of duality generalizes
the usual duality of planar graphs, exchanging vertices and faces. The
{\em reciprocal} of the hypermap $(\sigma,\alpha)$ is the hypermap
$(\alpha,\sigma)$. Taking the reciprocal generalizes taking the line
graph of a graph. Repeated use of taking the dual and the reciprocal
yields the following commutative diagram, which itself is a
generalization of Tutte's ``trinity''~\cite{Tutte} from graphs to hypermaps:

\[
\begin{tikzcd}[row sep=tiny]
{}  & (\alpha,\sigma)\arrow[dash]{r}{d}& (\sigma^{-1}\alpha,\sigma^{-1})\arrow[dash]{rd}{r} & \\
(\sigma,\alpha) \arrow[dash]{ru}{r} \arrow[dash,swap]{rd}{d} & & &
(\sigma^{-1},\sigma^{-1}\alpha) \\
  & (\alpha^{-1}\sigma,\alpha^{-1})\arrow[dash,swap]{r}{r} &
(\alpha^{-1},\alpha^{-1}\sigma)\arrow[dash,swap]{ru}{d} & \\
\end{tikzcd}  
\]

All hypermaps in the hexagonal diagram above have the same genus. 
Diagonally opposite to $(\sigma,\alpha)$ we find the {\em hyperdual} of
$(\sigma,\alpha)$, defined as $(\sigma^{-1},\sigma^{-1}\alpha)$. Note
that taking the hyperdual is also an involution. Besides the hypermaps
shown in the above hexagonal diagram, sometimes we will also consider 
the {\em mirrored hypermap} $(\sigma^{-1},\alpha^{-1})$ and the {\em
  Kreweras dual} $(\sigma, \alpha^{-1}\sigma)$ of the hypermap
$(\sigma,\alpha)$. The proof of the fact that the mirrored hypermap
$(\sigma^{-1},\alpha^{-1})$ has the same genus as $(\sigma,\alpha)$ is
left to the reader. The Kreweras dual of a noncrossing partition (that
is, a genus zero monopole) was introduced in~\cite{Kreweras}. The
Kreweras dual of a hypermap $(\sigma,\alpha)$ is
obtained from its hyperdual $(\sigma^{-1},\sigma^{-1}\alpha)$, by taking
its mirrored hypermap.

\subsection{Meanders, semimeanders, and stamp folding}
\label{sec:pmeanders}

Meanders and semimeanders have a vast literature, here we use the
terminology introduced in~\cite{DiFrancesco}. A {\em meander} of order
$n$ is a closed, self-avoiding loop, crossing a straight line at $2n$
points. We can visualize a meander as a closed, self-avoiding  walk  in
the plane crossing a river at $2n$ bridges. If we number the crossings
on the straight line left to right, and then list the crossings in the
order in which the meander encounters them, we obtain a {\em meandric
  permutation}. As pointed out by M.\ LaCroix~\cite{LaCroix}, before the
modern theory of enumerating meanders was developed,
Rosenstiehl has studied these permutations under the
name of {\em planar permutations}~\cite{Rosenstiehl}.  

A {\em semimeander} is a closed,
self-avoiding walk that crosses a half-line at $n$ points. As it is
stated in~\cite{DiFrancesco}, the number of semimeanders of order $n$
is the same as the number of {\em foldings $n-1$
stamps}~\cite{Lunnon,Touchard}. We will use the following equivalent
definition of stamp foldings.
\begin{definition}
  A {\em folding of $n-1$ stamps} is a permutation
  $\underline{\pi}=(\pi(1),\ldots, \pi(n))$
of the set $\{1,2,\ldots,n\}$, written as an ordered list, satisfying
the following conditions.
\begin{enumerate}
\item $\pi(1)=1$.
\item The list $\underline{\pi}$ does not contain any sublist of the
  form  $(2i,2j,2i+1,2j+1)$.
\item The list $\underline{\pi}$ does not contain any sublist of the
  form  $(2i-1,2j-1,2i,2j)$.
\end{enumerate}
\end{definition}
Arch diagrams representing stamp foldings are shown in the right half of
Figures~\ref{fig:semimeander}  and \ref{fig:semimeander2} in
Section~\ref{sec:semimeanders}. In both diagrams we see an ordered list
beginning with $1$. Each even label $2i$ is connected to the label
$2i+1$ by an upright arch (if $2i+1$ exists as a label), and each 
odd label $2i-1$ is connected to the label
$2i$ by an upside-down arch (if $2i$ exists as a label). Conditions (2)
and (3) are equivalent to stating that the resulting set of arcs does
not cross. 

\section{A consequence of Mach\`\i 's result and its interpretation}
\label{sec:Machi}

The starting point of our present investigation is the following result
of Mach\`\i~\cite{Machi}, generalizing a result of
Cori~\cite{Cori-hrec,Cori-Penaud}. 

\begin{theorem}
\label{thm:Machi}
Given a hypermap $(\sigma, \alpha)$, there is a bijection between the
genus $g$ unicellular hypermaps  $\theta$ spanning its hyperdual
$(\sigma^{-1},\sigma^{-1}\alpha)$, and the set
$C_{\sigma}(\sigma,\alpha)$, defined as the set of circular permutations 
$\zeta$ satisfying $g(\sigma,\zeta)=g(\sigma, \alpha)$ and
$g(\alpha,\zeta)=0$. The bijection is given by the rule $\theta\mapsto
\zeta=\sigma\theta$.
\end{theorem}
\begin{proof}
Although the above statement appears to be identical to Mach\`\i 's
result, we need to check it is still valid, even though we multiply
permutations right to left, whereas he multiplies them left to right.

We begin with translating Mach\`\i 's result into right to left
multiplication form: given  a hypermap $(\sigma, \alpha)$, there is a
bijection between the spanning  genus $g$ unicellular hypermaps
$\theta'$ of the hypermap $(\sigma^{-1},\alpha\sigma^{-1})$, and the set
$C_{\sigma}(\sigma,\alpha)$, defined as the set of circular permutations 
$\zeta$ satisfying $g(\sigma,\zeta)=g(\sigma, \alpha)$ and
$g(\alpha,\zeta)=0$. The bijection is given by the rule $\theta'\mapsto
\zeta=\theta'\sigma$. Note that the definition of
$C_{\sigma}(\sigma,\alpha)$ remains unchanged, as the definition of the
genus depends only on counting cycles in a way that is independent of
the direction of the multiplication, as a consequence of the identity
\begin{equation}
z(\alpha\beta)=z(\beta\alpha)  
\end{equation}
which is a direct consequence of the fact that
$\alpha\beta=\beta^{-1}(\beta\alpha)\beta$ is a conjugate of
$\beta\alpha$. The hypermap $(\sigma^{-1},\alpha\sigma^{-1})$ is not the
hyperdual of $(\sigma,\alpha)$ in our terminology, but it is isomorphic
to it:  $\alpha\sigma^{-1}=\sigma(\sigma^{-1}\alpha)\sigma^{-1}$ is a
conjugate of $\sigma^{-1}\alpha$ and sending each $i$ into $\sigma(i)$
induces map from $(\sigma^{-1},\sigma^{-1}\alpha)$ to
$(\sigma^{-1},\alpha\sigma^{-1})$ that is an isomorphism of
hypermaps. As a consequence, sending each spanning unicellular hypermap
$\theta$ of $(\sigma^{-1},\sigma^{-1}\alpha)$ into
$\theta'=\sigma\theta\sigma^{-1}$ we 
obtain a bijection between the spanning unicellular hypermaps of
$(\sigma^{-1},\sigma^{-1}\alpha)$ and the spanning unicellular hypermaps of
$(\sigma^{-1},\alpha\sigma^{-1})$. Composing the map $\theta\mapsto
\theta'=\sigma\theta\sigma^{-1}$  with the map $\theta'\mapsto
\zeta=\theta'\sigma$ we obtain the desired bijection $\theta\mapsto
\zeta=\sigma\theta\sigma^{-1}\sigma=\sigma\theta$.
\end{proof}

By replacing the hypermap $(\sigma,\alpha)$ by its Kreweras dual
$(\sigma,\alpha^{-1}\sigma)$ in Theorem~\ref{thm:Machi} we obtain
the following consequence.
\begin{corollary}
\label{cor:MCP} There is a bijection between the spanning genus $g$
unicellular hypermaps $\theta$ of a hypermap $(\sigma,\alpha)$ of genus
$g$ and the set 
$$C_{\sigma}(\sigma,\alpha^{-1}\sigma)=\{\zeta\::\: z(\zeta)=1,
g(\sigma,\zeta)=g(\sigma,\alpha^{-1}\sigma),
g(\alpha^{-1}\sigma,\zeta)=0\},$$
taking each spanning unicellular hypermap
$\theta$ into $\zeta=\theta^{-1}\sigma$. 
\end{corollary}  

Note that $\zeta=\theta^{-1}\sigma$ is the permutation whose only cycle is the
only face of the hypermap $(\sigma,\theta)$. The hypermap
$(\sigma,\theta^{-1}\sigma)$ is the Kreweras dual of $(\sigma,\theta)$,
requiring $g(\sigma,\zeta)=g(\sigma,\alpha^{-1}\sigma)$ is equivalent to
requiring that the spanning unicellular hypermap $(\sigma,\theta)$ must
have the same genus as $(\sigma,\alpha)$. Less immediate is the
following consequence:  
visiting the labels in the order of
$\zeta$ amounts to traversing the only face of $(\sigma,\theta)$
according to its orientation. During this traversal we visit the
faces of $(\sigma,\alpha)$ in their cyclic order in 
$\alpha^{-1}\sigma$. Least obvious is the fact that the above criteria on the
genuses and the number of cycles {\em characterize} the only faces of
spanning unicellular hypermaps. 

\begin{figure}[h]
\begin{center}
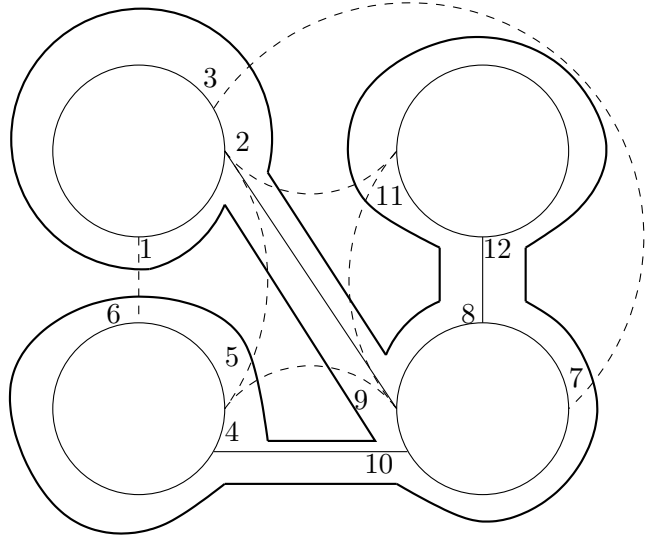
\end{center}
\caption{A spanning hypertree of the hypermap shown in Fig.\ref{fig:hypermap}}
\label{fig:stree}
\end{figure}

\begin{example}
\label{e:stree}  
Consider the hypermap $(\sigma,\alpha)$ shown in
Fig.~\ref{fig:hypermap}. The permutation
$\theta=(1)(2,9)(3)(4,10)(5)(6)(7)(8,12)(11)$  is
a refinement of $\alpha$, and the hypermap
$(\sigma,\theta)$ is a spanning hypertree of the hypermap $(\sigma,\alpha)$.
 This spanning hypertree, together with its only face
 $\theta^{-1}\sigma=(1,9,4,5,6,10,7,12,11,8,2,3)$ is shown in
 Fig.~\ref{fig:stree}.  
\end{example}  
It should not confuse the reader that the spanning hypertree in
Example~\ref{e:stree} is also a tree. We include a second example of a
spanning hypertree: this example is not a tree, and the hypermap is also
nonplanar.  
\begin{example}
\label{e:robert}    
Consider the hypermap
$$(\sigma,\alpha)=((1,4,7)(2,5,8)(3,6,9),(1,2,3)(4,5,6)(7,8,9))$$
shown in  Fig.~\ref{fig:robert}. It has genus $1$, hence we draw it on a torus: 
the reader is supposed to identify pairs of points having the same first
or second coordinate on the boundary of the bounding box. The circles
represent the vertices, the hyperedges are bounded by dashed lines,
except for the cycle $(1,2,3)$. The permutation $\theta=(1,2,3)$ is a
refinement of $\alpha$ and $(\sigma,\theta)$ is a spanning
hypertree. The only face of $(\sigma,\theta)$ is
$\theta^{-1}\sigma=(1,4,7,3,6,8,2,5,8)$.  
\end{example}  
\begin{figure}[h]
\begin{center}
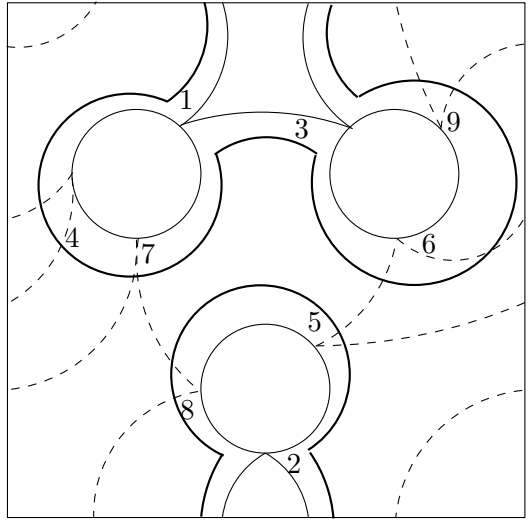
\end{center}
\caption{A hypermap of genus $1$ and its spanning hypertree}
\label{fig:robert}
\end{figure}

We may use the cyclic order of $\theta^{-1}\sigma$ to obtain a special
representation of the spanning hypertree of a genus zero hypermap, which we
call the {\em one-line diagram $D(\sigma,\alpha,\theta)$}. We list the
labels in the 
order they appear in $\theta^{-1}\sigma$, starting with $1$, in the
left to right order, below a horizontal line. Above the horizontal line
we connect the elements that are adjacent in a cycle of
$\alpha^{-1}\sigma$, below the horizontal line we connect the labels
that are adjacent in a cycle of $\sigma$, and we shade the regions
representing the cycles of these permutations, as shown in
Fig.~\ref{fig:1line}. We make sure that the labels are recorded on the
left hand side of the points corresponding to them on the horizontal
line. 

The following statement is an obvious 
consequence of the definition of $C_{\sigma}(\sigma,\alpha^{-1}\sigma)$.
\begin{proposition}
  \label{prop:1line0}
Let $(\sigma,\alpha)$ be a genus zero hypermap on the set $\{1,2,\ldots,n\}$.  
The circular permutation of $\{1,2,\ldots,n\}$ is the only face of a
spanning hypertree of $(\sigma,\alpha)$ if and only if its one-line
diagram satisfies the following. 
\begin{enumerate}
\item The arcs above the horizontal line represent the parts of a noncrossing
  partition whose parts are the cycles of $\alpha^{-1}\sigma$, that is,
  the faces.
\item The arcs below the horizontal line represent the parts of a noncrossing
  partition whose parts are the cycles of $\sigma$, that is,
  the vertices.  
\end{enumerate}  
\end{proposition}
For hypermaps of higher genus a similar characterization may be
formulated, which we will present in Section~\ref{sec:hyperdc}.
\begin{figure}[h]
\begin{center}
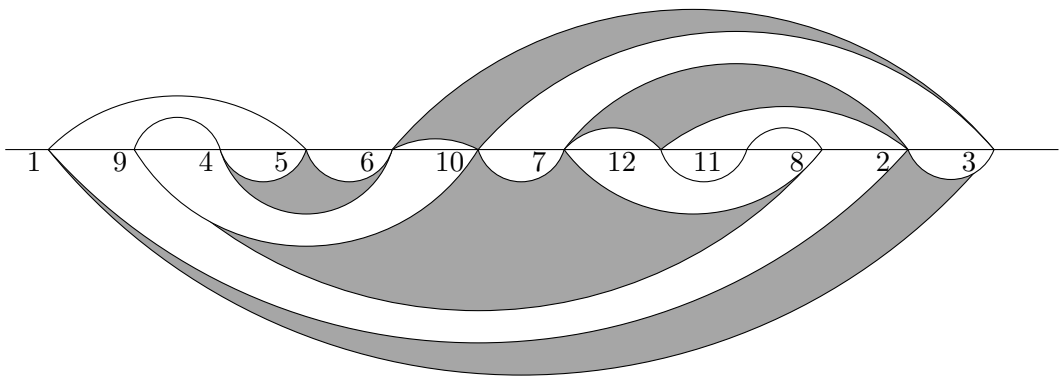
\end{center}
\caption{One-line representation of the spanning hypertree shown in Fig.\ref{fig:stree}}
\label{fig:1line}
\end{figure}
It is worth noting that the regions between the shaded areas allow us to
read off the hyperedges of $(\sigma,\alpha)$ and $(\sigma,\theta)$
respectively. 
\begin{theorem}
Let $(\sigma,\alpha)$ be a hypermap of genus zero and $(\sigma,\theta)$
a spanning hypertree of $(\sigma,\alpha)$. Disregarding the horizontal
line, the unshaded regions between and around the regions representing
the cycles of $\sigma$ and $\sigma^{-1}\alpha$ in
$D(\sigma,\theta,\alpha)$ contain the labels of
the hyperedges of $(\sigma,\alpha)$. The parts of these regions below
the horizontal line are labeled with the hyperedges of $(\sigma,\theta)$.
\end{theorem}  
\begin{proof} (Sketch.)
The first statement is a direct consequence of the definition. To prove
the second statement, observe that going around the spanning hypertree,
on one side we have the vertices and the hyperedges of $(\sigma,\theta)$
and the vertices and on the other side we have the faces and parts of
the hyperedges of $(\sigma,\alpha)$ which connect the cycles of $\theta$
into cycles of $\alpha$
\end{proof}  

\section{Hyperdeletions, hypercontractions and compatible pairs of
  tours}
\label{sec:hyperdc}

Recall that a transposition $\tau=(i,j)$ {\em connects} a permutation $\pi$
if $i$ and $j$ belong to different cycles of $\pi$, otherwise we say that
it {\em disconnects} $\pi$. The reason behind this terminology is {\em
  Serret's lemma}, according to which for a transposition $\tau$ connecting
$\pi$ we have $z(\tau\pi)=z(\pi\tau)=z(\pi)-1$ and for a transposition
$\tau$ disconnecting $\pi$ we have $z(\tau\pi)=z(\pi\tau)=z(\pi)+1$. For
example, when $\tau$ disconnects $\pi$, the cycles of $\pi\tau$ and the
cycles of $\tau\pi$ are obtained from the cycles of $\pi$ by breaking
the single cycle containing both $i$ and $j$ into two cycles, one
containing $i$, one containing $j$. Keeping
this in mind, we make the following definition.

\begin{definition}
A {\em hyperdeletion} is the operation of replacing a hypermap
$(\sigma,\alpha)$ with the hypermap 
$(\sigma,\alpha\delta)$ where $\delta=(i,j)$ is a transposition
disconnecting $\alpha$. We call the hyperdeletion {\em topological} if
$\delta$ also connects $\alpha^{-1}\sigma$, that is,
$z(\delta\alpha^{-1}\sigma)=z(\alpha^{-1}\sigma)-1$. In short, we will
say that {\em $\delta$ is a hyperdeletion for $(\sigma,\alpha)$}  if
the operation $(\sigma,\alpha)\mapsto (\sigma,\alpha\delta)$ is a
hyperdeletion. 
\end{definition}
It is part of the definition of a hyperdeletion that
$(\sigma,\alpha\delta)$ must still be a hypermap. We are simply not
allowed to replace $\alpha$ with $\alpha\delta$ if $\sigma$ and
$\alpha\delta$ do not generate a transitive group. In that case we will
say that $\delta$ is an {\em isthmus}. Note that this definition
coincides with the usual definition if $(\sigma,\alpha)$ is {\em a map}
in which each cycle of $\alpha$ has at most two elements, equivalently,
$\alpha=\alpha^{-1}$ holds. In that case,
disconnecting a cycle of $\alpha$, while keeping $\sigma$ unchanged,
corresponds to deleting an edge. For topological deletions, the fact
that $z(\delta\alpha^{-1}\sigma)=z(\alpha^{-1}\sigma)-1$ automatically 
guarantees that $(\sigma,\alpha\delta)$ is also a hypermap, this was
already noted in~\cite{Cori-Machi} and \cite{Machi}. Note that, for all
hyperdeletions, the faces of
$(\sigma,\alpha\delta)$ are the cycles of $\delta\alpha^{-1}\sigma$. 
By Serret's lemma, for topological hyperdeletions we have  
$z(\alpha\delta)=z(\alpha)+1$ and
$z(\alpha^{-1}\sigma)=z(\delta\alpha^{-1}\sigma)+1$, hence the hypermap
$(\sigma,\alpha\delta)$ has the same genus as $(\sigma,\alpha)$. If the
hyperdeletion is not topological then we have $z(\alpha\delta)=z(\alpha)+1$ and
$z(\alpha^{-1}\sigma)=z(\delta\alpha^{-1}\sigma)-1$, and we have
$g(\sigma,\alpha\delta)=g(\sigma,\alpha)-1$. As a consequence,
non-topological hyperdeletions exist only for hypermaps of positive
genus, and they decrease the genus of the hypermap by one.

Dually, we define {\em hypercontractions} as follows.
\begin{definition}
A {\em hypercontraction} is the operation of replacing a hypermap
$(\sigma,\alpha)$ with the hypermap $(\gamma\sigma,\gamma\alpha)$ where
$\gamma=(i,j)$ is a transposition disconnecting $\alpha$. We call a
hypercontraction {\em topological} if it also connects $\sigma$, that
is, we have $z(\gamma\sigma)=z(\sigma)-1$. In short, we will
say that {\em $\gamma$ is a hypercontraction for $(\sigma,\alpha)$}  if
the operation $(\sigma,\alpha)\mapsto (\gamma\sigma,\gamma\alpha)$ is a
hypercontraction.  
\end{definition}
Once again, we require $(\gamma\sigma,\gamma\alpha)$ to be a hypermap,
and we do not allow the operation $(\sigma,\alpha)\mapsto
(\gamma\sigma,\gamma\alpha)$ if $\gamma\sigma$ and $\gamma\alpha$
generate a non-transitive subgroup.

The description of non-topological hypercontractions is intuitively less
obvious, even for maps. Clearly $\gamma=(i,j)$ is a topological
hypercontraction if an only if $i$ and $j$ belong to different cycles of
$\sigma$. If $i$ and $j$ belong to the same cycle of $\sigma$ in a map
$(\sigma,\alpha)$ then $(i,j)$ is a loop and
$(\gamma\sigma,\gamma\alpha)$ is obtained by deleting this loop and
{\em splitting} the vertex incident to this loop into two vertices.
\begin{definition}
\label{def:split}
Given an undirected multigraph $G$, we define a {\em vertex splitting}
as an operation that replaces a vertex $v$ of $G$ with two vertices
$v_1$ and $v_2$, and that replaces each edge $e$ connecting some vertex
$u$ with $v$ as follows: if $u\neq v$ then $e$ is replaced with an edge
connecting $u$ with $v_1$ or $v_2$, if $u=v$ then the loop edge incident
to $v$ is replaced with a loop edge incident to $v_1$ or $v_2$ or by an edge
connecting $v_1$ with $v_2$. If the graph $H$ is obtained from $G$ by
vertex splitting, we say that $G$ is obtained from $H$ by {\em vertex
  merging}. For a hypermap $(\sigma,\alpha)$, a {\em topological vertex
  splitting} is a map $(\sigma,\alpha)\mapsto ((i,j)\sigma,(i,j)\alpha)$
where $i$ and $j$ belong to the same cycle of $\sigma$.    
\end{definition}
\begin{remark}  
Various definitions of a vertex splitting
exist in the literature of graph theory, and the one given in
Definition~\ref{def:split} above does not seem to be the most frequently
used one. That said, this definition of a vertex splitting is used for
example in~\cite{Eades-dM}. 
\end{remark}
Note that we allow the use of a non-topological contraction only if the
resulting graph is not disconnected. For example for $\gamma=(1,2)$ and
the monopole $(\sigma,\alpha)=((1,2,3,4),(1,2)(3,4))$ the pair of
permutations $(\gamma\sigma,\gamma\alpha)=((1)(2,3,4), (1)(2)(3,4))$ is
not a map: it has two isolated vertices, one of them is incident to the
loop $(3,4)$. On the other hand, for $\gamma=(1,3)$ and the monopole
$(\sigma,\alpha)=((1,2,3,4),(1,3)(2,4))$ the pair of permutations
$(\gamma\sigma,\gamma\alpha)=((1,2)(3,4),(1)(2,4)(3)$ is a map: we
obtain two vertices connected by the edge $(2,4)$.  

A topological hypercontraction does not disconnect the hypermap, nor
does it change the genus, whereas a non-topological hypercontraction
may disconnect a hypermap, and if it does not then it decreases the
genus by one. Hyperdeletions and hypercontractions are 
dual to each other in the following sense.

\begin{lemma}
\label{lem:deletion_contraction}  
Given a hypermap $(\sigma,\alpha)$ and a transposition $\tau$
disconnecting $\alpha$, the operation  
$(\sigma,\alpha)\mapsto (\sigma,\alpha\tau)$ is a hyperdeletion if and
only if the operation 
$(\alpha^{-1}\sigma,\alpha^{-1})\mapsto
(\tau\alpha^{-1}\sigma,\tau\alpha^{-1})$ is a hypercontraction.
If one of these operations is topological then so is the other one. 
\end{lemma}  
The straightforward verification is left to the reader. It is worth
memorizing that hyperdeletions keep the vertices unchanged and change
the number of faces, whereas hypercontractions keep faces unchanged and
change the number of vertices.  In analogy to~\cite[Lemma 3]{Machi} we
may note the following.
\begin{lemma} If $\delta$ is a deletion ($\gamma$ is a contraction) for
  the hypermap $(\sigma,\alpha)$ then the permutation $\alpha\delta$
  (the permutation $\gamma\alpha$) is a refinement of $\alpha$. 
\end{lemma}
To allow repeated use of hyperdeletions and hypercontractions, we make the
following definition.
\begin{definition}
Given a hypermap $(\sigma,\alpha)$ we call a sequence of transpositions
$\delta_1,\ldots,\delta_k$ a {\em sequence of hyperdeletions for
  $(\sigma,\alpha)$} if $\delta_1$ is a hyperdeletion for
  $(\sigma,\alpha)$ and for each $i$ satisfying $1<i\leq k$ the
  transposition $\delta_i$ is a hyperdeletion for
  $(\sigma,\alpha\delta_1\cdots \delta_{i-1})$. Dually, we call a
  sequence of transpositions $\gamma_1,\ldots,\gamma_k$ a {\em sequence
    of hypercontractions for $(\sigma,\alpha)$} if $\gamma_1$ is a
    hypercontraction for $(\sigma,\alpha)$ and for each $i$ satisfying
    $1<i\leq k$ the transposition $\gamma_i$ is a hypercontraction for
  $(\gamma_{i-1}\cdots \gamma_1\sigma,\gamma_{i-1}\cdots \gamma_1\alpha)$.   
\end{definition}  
In analogy to~\cite[Lemma 1]{Machi}, one may note that
we can always find a topological hyperdeletion $\delta$ decreasing the number of
faces of a hypermap, as long as $(\sigma,\alpha)$ is not
unicellular. Indeed, if all transpositions connecting
$\alpha^{-1}\sigma$ also connect $\alpha$ then the cycles of 
$\alpha$ are contained in the cycles of $\alpha^{-1}\sigma$, the
permutations $\alpha$ and $\alpha^{-1}\sigma$ do not generate a
transitive permutation group, the same holds for $\alpha$ and
$\sigma$, in contradiction with $(\sigma,\alpha)$ being a hypermap. Thus
we obtain the following statement.
\begin{lemma}
\label{lem:deletions}  
For any hypermap $(\sigma,\alpha)$ with more than one face, there is a
sequence of topological hyperdeletions
$\delta_1,\ldots,\delta_{z(\alpha^{-1}\sigma)-1}$ of length
$z(\alpha^{-1}\sigma)-1$. For any sequence of hyperdeletions of such
length, the hypermap
$(\sigma,\alpha\delta_1\cdots\delta_{z(\alpha^{-1}\sigma)-1})$ is a
spanning unicellular hypermap of $(\sigma,\alpha)$ of genus
$g(\sigma,\alpha)$.     
\end{lemma}
 The following strengthening of Lemma~\ref{lem:deletions} is implicit in
the proof of~\cite[Theorem 1]{Machi}.
\begin{proposition}
\label{prop:deletions}  
Given a hypermap $(\sigma,\alpha)$ of genus $g$ and a spanning 
genus $g$ unicellular hypermap $(\sigma,\theta)$ of it, there is a sequence of 
topological hyperdeletions of length $z(\alpha^{-1}\sigma)-1$ taking
$(\sigma,\alpha)$ into $(\sigma,\theta)$. Conversely, applying any
sequence of topological hyperdeletions of length $z(\alpha^{-1}\sigma)-1$ to
$(\sigma,\alpha)$ yields a spanning genus $g$ unicellular hypermap of
$(\sigma,\alpha)$.  
\end{proposition}
The key idea behind proving Proposition~\ref{prop:deletions} is that
refining $\alpha$ to $\theta$ amounts to replacing each cycle of
$\alpha$ with a noncrossing partition with respect to its cyclic order,
and noncrossing partitions are intersections of the coatoms in the
noncrossing partition lattice. Each coatom had exactly two parts,
replacing a single cycle with two noncrossing cycles amounts to applying
a topological hyperdeletion.

Applying Proposition~\ref{prop:deletions} to the dual hypermap we obtain
the following.
\begin{corollary}
\label{cor:dual_deletions}  
Given a hypermap $(\sigma,\alpha)$ of genus $g$, the set of all spanning
genus genus $g$ unicellular hypermaps of its dual hypermap 
$(\alpha^{-1}\sigma,\alpha^{-1})$ are exactly the hypermaps obtained by
by applying a sequence of topological hyperdeletions  $\gamma_{1}, \gamma_{2},
\cdots, \gamma_{z(\sigma)-1}$ to $(\alpha^{-1}\sigma,\alpha^{-1})$.
\end{corollary}
By Lemma~\ref{lem:deletion_contraction},
Corollary~\ref{cor:dual_deletions} may be rephrased as follows.
\begin{corollary}
\label{cor:contractions}  
Given a hypermap $(\sigma,\alpha)$ of genus $g$, the spanning
genus $g$ unicellular hypermaps of its dual
$(\alpha^{-1}\sigma,\alpha^{-1})$ are exactly the duals of all hypermonopoles of the form $(\gamma\sigma,\gamma\alpha)$ where
$\gamma=\gamma_{z(\sigma)-1}\gamma_{z(\sigma)-2}\cdots\gamma_{1}$,
and $\gamma_{1}, \gamma_{2}, \cdots, \gamma_{z(\sigma)-1}$ is 
a sequence of topological hypercontractions for $(\sigma,\alpha)$.
\end{corollary}
Note that the dual of the hypermonopole $(\gamma\sigma,\gamma\alpha)$ is
$(\alpha^{-1}\sigma,\alpha^{-1}\gamma^{-1})$. By
Theorem~\ref{thm:dhtree}, the hypermap
$(\alpha^{-1}\sigma,\alpha^{-1}\gamma^{-1})$ is a spanning genus $g$
unicellular hypermap of $(\alpha^{-1}\sigma,\alpha^{-1})$ if
and only if $(\sigma,\gamma^{-1})$ is a spanning hypertree of
$(\sigma,\alpha)$.   

To summarize, we obtain the following result.
\begin{theorem}
\label{thm:0trees}  
Given a hypermap $(\sigma,\alpha)$, a hypermonopole may be obtained from
it by a sequence of topological hypercontractions if and only if it is
of the form $(\gamma\sigma,\gamma\alpha)$ where $(\sigma,\gamma^{-1})$ is any
spanning hypertree of $(\sigma,\alpha)$.
\end{theorem}

Applying Corollary~\ref{cor:MCP} to the spanning genus $g$ unicellular
hypermaps of the dual of a hypermap of genus~$g$ we obtain the following
statement. 

\begin{proposition}
\label{prop:vertextour}  
For a hypermap $(\sigma,\alpha)$ there is a bijection
between the spanning hypertrees of $(\sigma,\alpha)$
and the set
$$C_{\alpha^{-1}\sigma}(\alpha^{-1}\sigma,\sigma)=\{\eta\::\: z(\eta)=1,
g(\alpha^{-1}\sigma,\eta)=g(\sigma,\alpha),
g(\sigma,\eta)=0\},$$
taking each spanning hypertree
$(\sigma,\gamma^{-1})$ of $(\sigma,\alpha)$ into $\eta=\gamma\sigma$.
\end{proposition}
\begin{proof}
Each spanning hypertree $(\sigma,\gamma^{-1})$ of $(\sigma,\alpha)$
corresponds to a spanning genus $g$ unicellular hypermap
$(\alpha^{-1}\sigma,\alpha^{-1}\gamma^{-1})$ of the dual hypermap  
$(\alpha^{-1}\sigma,\alpha^{-1})$. The map described in
Corollary~\ref{cor:MCP} sends
$(\alpha^{-1}\sigma,\alpha^{-1}\gamma^{-1})$ into 
$\eta=(\alpha^{-1}\gamma^{-1})^{-1}\alpha^{-1}\sigma=\gamma\sigma$. The
range of the map $(\sigma,\gamma^{-1})\mapsto \gamma\sigma$ is
$$C_{\alpha^{-1}\sigma}(\alpha^{-1}\sigma,\sigma)
=\{\eta\::\: z(\eta)=1,
g(\alpha^{-1}\sigma,\eta)=g(\alpha^{-1}\sigma,\sigma), g(\sigma,\eta)=0\}. 
$$
Note finally that $g(\alpha^{-1}\sigma,\sigma)=g(\sigma,\alpha)$
is a direct consequence of~(\ref{eq:genusdef}) and
  $z(\sigma^{-1}\alpha^{-1}\sigma)=z(\alpha)$. 
\end{proof}  

A particularly important instance of Proposition~\ref{prop:vertextour}
is the case of maps. For these hyperdeletions and hypercontraction are
actual deletions and contractions, as defined in graph theory. A
sequence of deletions or contractions is a list of pairwise disjoint
transpositions, their order does not matter, only the subgraph formed by
the deleted, respectively contracted edges. As a consequence a sequence
of contractions $\gamma_1,\ldots,\gamma_k$ yields an involution
$\gamma=\gamma_k\cdots \gamma_1$ which is its own inverse. We obtain the
following corollary.

\begin{corollary}
\label{cor:vertextour}  
For a map $(\sigma,\alpha)$ there is a bijection
between the spanning trees of the underlying graph of $(\sigma,\alpha)$
and the set
$$C_{\alpha^{-1}\sigma}(\alpha^{-1}\sigma,\sigma)=\{\eta\::\: z(\eta)=1,
g(\alpha^{-1}\sigma,\eta)=g(\sigma,\alpha),
g(\sigma,\eta)=0\},$$
taking each spanning tree
$(\sigma,\gamma)$ into $\eta=\gamma\sigma$.
\end{corollary}   

We wish to point out that Corollary~\ref{cor:vertextour}
implies Bernardi's~\cite[Lemma 4.2]{Bernardi-embeddings} stating that the
{\em tour} of a spanning tree of an embedded topological graph is always
a cyclic permutation. For each spanning tree $(\sigma,\gamma)$ we may
define the inverse of $\eta=\gamma\sigma$ as the {\em vertex tour} of
the spanning tree. We illustrate this observation with the next
example. 

\begin{example}
The map $(\sigma,\alpha)$, given by $\sigma=(1,4,2,12)(8,11,9)(5,7,3,6)(10)$
and $\alpha=(1,7)(2,8)(3,9)(4,10),(5,11)(6,12)$, and shown in
Figure~\ref{fig:bernardi}, is isomorphic as an embedded topological
graph to an example of Bernardi. (It is drawn as the mirror image of the
embedded topological graph on the left hand side of~\cite[Figure
  1]{Bernardi-embeddings}. Bernardi's letters $a$ through $f$ correspond to the
numbers $1$ through $6$, while the letters $a'$ through $f'$ correspond
to the numbers $7$ trough $12$ in our picture. We took the mirror image
of Bernardi's figure to match his counterclockwise listing of vertex
labels with our clockwise listing. Bernardi gives an example of a
spanning tree~\cite[Figure 4]{Bernardi-embeddings}, which corresponds to
$\gamma=(1,7)(2,8)(4,10)$, marked in bold on the left hand side of
Figure~\ref{fig:bernardi}. The order of the transpositions does not
matter, contracting these edges results in a monopole. The permutation
$$
\eta=\gamma\sigma=(1,10,4,8,11,9,2,12,7,3,6,5)\in
C_{\alpha^{-1}\sigma}(\alpha^{-1}\sigma,\sigma) 
$$
corresponds to the list $(a,d',d,b',e',c',b,f',a',c,f,e)$ in Bernardi's
notation which is the inverse of the tour of the graph (or
{\em motion function}) defined by Bernardi~\cite[p.\ 145]{Bernardi-embeddings}.
\end{example}

\begin{figure}[h]
\begin{center}
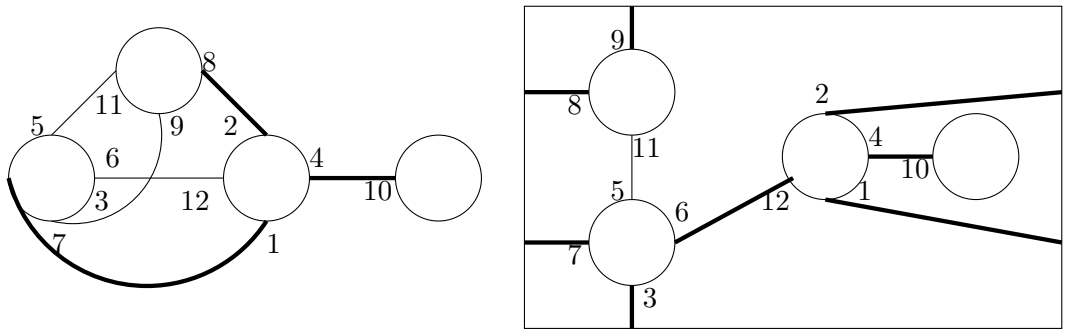
\end{center}
\caption{Bernardi's example}
\label{fig:bernardi} 
\end{figure}

The map shown in Figure~\ref{fig:bernardi} has genus $1$, and we may
draw it on the torus without crossing edges, as shown on the right hand
side of the figure. (The rules of identifying pairs of points on the
boundary of the bounding box are the same as in
Figure~\ref{fig:robert}.) The spanning tree on the left 
may be extended to a spanning genus $1$ unicellular hypermap by adding
the edges $(3,9)$ 
and $(6,12)$, also shown in bold on the right hand side. The spanning
genus $1$ unicellular hypermap $\theta=(1,7)(2,8)(3,9)(4,10)(6,12)$
corresponds to 
$$
\zeta=\theta^{-1}\sigma=(1,10,4,8,11,3,12,7,9,2,6,5)\in
C_{\sigma}(\sigma,\alpha^{-1}\sigma) 
$$
via the bijection described in Corollary~\ref{cor:MCP}. This too is a
{\em tour} of a spanning tree, albeit it is the tour of a spanning tree
of genus $1$. This observation inspires the following definition.

\begin{definition}
\label{def:Bernardi_tour}
Given a hypermap $(\sigma,\alpha)$ of genus $g$, for each spanning
hypertree $(\sigma,\gamma^{-1})$ of $(\sigma,\alpha)$ we call
the cyclic permutation $\eta=\gamma\sigma$ the {\em vertex tour} of the
spanning hypertree, and for each spanning 
genus $g$ unicellular hypermap $(\sigma,\theta)$ we call the cyclic permutation
$\zeta=\theta^{-1}\sigma$ the {\em face tour} of
the spanning genus $g$ unicellular hypermap.   
\end{definition}  
By Corollary~\ref{cor:MCP} the set of face tours is exactly the set
$C_{\sigma}(\sigma,\alpha)$. If we write the labels in the cyclic order
of a face tour $\zeta$ and represent the faces as in the upper half of
Figure~\ref{fig:1line}, we get a noncrossing representation of the faces
since, by the definition of $C_{\sigma}(\sigma,\alpha)$, we have
$g(\alpha^{-1}\sigma,\zeta)=0$. Similarly, by Corollary~\ref{prop:vertextour} the
set of vertex tours is exactly the set 
$C_{\alpha^{-1}\sigma}(\alpha^{-1}\sigma,\sigma)$. If we write the
labels in the cyclic order 
of a face tour $\eta$ and represent the vertices as in the lower half of
Figure~\ref{fig:1line}, we get a noncrossing representation of the
vertices. The notion of a face tour and a vertex tour coincides for
genus zero hypermaps.   

The relation between the spanning genus $1$ unicellular hypermap and
spanning tree in Figure~\ref{fig:bernardi} and their associated face-
and vertex tours may be also interpreted as follows. We begin with a
hypermap $(\sigma,\alpha)$ of genus $g$. We perform a sequence of
topological hyperdeletions
$\delta_1,\ldots,\delta_{|\alpha^{-1}\sigma|-1}$ and obtain a 
spanning genus $g$ unicellular hypermap $(\sigma,\alpha\delta)$ of
$(\sigma,\alpha)$ with the single face $\zeta=\delta^{-1}\alpha^{-1}\sigma\in
C_{\sigma}(\sigma,\alpha^{-1}\sigma)$. These topological hyperdeletions
leave the vertices $\sigma$ unchanged. Next we perform a sequence of
topological hypercontractions $\gamma_1,\ldots,\gamma_{|\sigma|-1}$ and obtain a
monopole $(\gamma\sigma,\gamma\alpha\delta)$ with the single vertex
$\eta=\gamma\sigma\in C_{\alpha^{-1}\sigma}(\alpha^{-1}\sigma,\sigma)$
and single face $\zeta$. Here $\gamma=\gamma_{|\sigma|-1}\cdots\gamma_1$
and $(\sigma,\gamma^{-1})$ is a spanning hypertree of
$(\sigma,\alpha)$. The vertex $\eta$ is shown as a dashed circle in
the middle of Figure~\ref{fig:bernardi2}, and the face $\zeta$ appears as a
collection of four line segments (also to be interpreted as a circle on
the torus) on the same illustration. Inside the circle representing
$\eta$ we find the cycles of $\sigma$, represented by the shaded
areas. Inside the circle represented by $\zeta$ (shown in the four
corners of Figure~\ref{fig:bernardi2}) we find the cycles of
$\alpha^{-1}\sigma=(1,10,4,8,5)(2,6,11,3,12,7,9)$, represented by the
shaded areas.  Between the two circles we have the hyperedges of
$(\eta,\eta\zeta^{-1})=(\gamma\sigma,\gamma\theta)$.

\begin{figure}[h]
\begin{center}
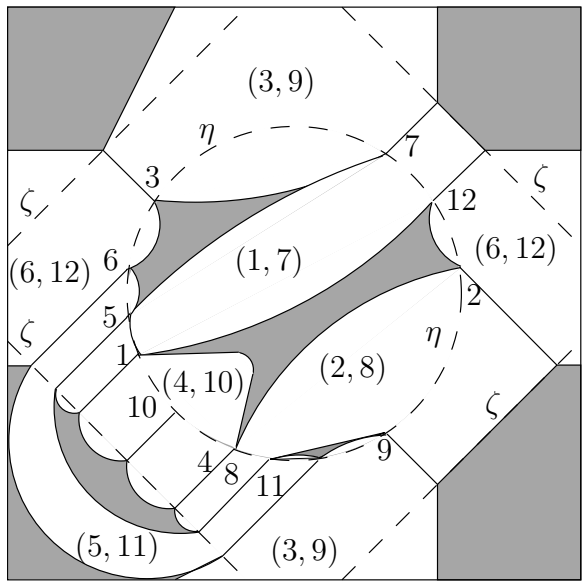
\end{center}
\caption{A $2$-disk diagram of Bernardi's example}
\label{fig:bernardi2} 
\end{figure}

Wanting to generalize this picture to all hypermaps, we
make the following definition.
\begin{definition}
Given a hypermap $(\sigma, \alpha)$, we say that the pair of circular
permutations $(\eta,\zeta)$ satisfying $\eta\in
C_{\alpha^{-1}\sigma}(\alpha^{-1}\sigma,\sigma)$ and $\zeta\in
C_{\sigma}(\sigma,\alpha)$ is {\em a compatible pair of tours} if
$(\eta,\eta\zeta^{-1})$ is a hypermap of genus $g(\sigma,\alpha)$.
\end{definition} 

\begin{remark}
Note that the only vertex of $(\eta,\eta\zeta^{-1})$ is the only cycle
of $\eta$ and the only face is the only cycle of
$(\eta\zeta^{-1})^{-1}\eta=\zeta$.  
\end{remark}  

\begin{theorem}
Given a hypermap $(\sigma, \alpha)$, the pair of circular
permutations $(\eta,\zeta)$ satisfying $\eta\in
C_{\alpha^{-1}\sigma}(\alpha^{-1}\sigma,\sigma)$ and $\zeta\in
C_{\sigma}(\sigma,\alpha)$ is compatible if and only if there is a
sequence of topological hyperdeletions
$\delta_1,\ldots,\delta_{|\alpha^{-1}\sigma|-1}$ followed by a sequence
of topological hypercontractions $\gamma_1,\ldots,\gamma_{|\sigma|-1}$ taking
$(\sigma, \alpha)$ into $(\eta,\eta\zeta^{-1})$. 
\end{theorem}
\begin{proof}
Assume first that there is a sequence of topological hyperdeletions
$\delta_1,\ldots,\delta_{|\alpha^{-1}\sigma|-1}$ followed by a sequence
of hypercontractions $\gamma_1,\ldots,\gamma_{|\sigma|-1}$ taking
$(\sigma, \alpha)$ into $(\eta,\eta\zeta^{-1})$. Introducing
$\delta=\delta_1\cdots \delta_{|\alpha^{-1}\sigma|-1}$ and
$\gamma=\gamma_{|\sigma|-1}\cdots\gamma_1$, the hypermap obtained from
$(\sigma,\alpha)$ after performing these topological hyperdeletions and
hypercontractions is $(\gamma\sigma,\gamma\alpha\delta)$. Here
$\eta=\gamma\sigma$ and the equality $\eta\zeta^{-1}=\gamma\alpha\delta$
implies
$\zeta=(\eta\zeta^{-1})^{-1}\eta=(\gamma\alpha\delta)^{-1}\gamma\sigma=\delta^{-1}\alpha^{-1}\sigma$. Since
topological hyperdeletions and hypercontractions do not change the
genus, we have 
$g(\eta,\eta\zeta^{-1})=g(\sigma,\alpha)$ and $(\eta,\zeta)$ is a
compatible pair of circular permutations.

Conversely, assume that $(\eta,\zeta)$ is a
compatible pair of tours. By definition, $\zeta$ belongs
to $C_{\sigma}(\sigma,\alpha)$, and by Corollary~\ref{cor:MCP} there is
a spanning genus $g$ unicellular hypermap $(\sigma,\theta)$ of $(\sigma,\alpha)$
satisfying $\theta^{-1}\sigma=\zeta$. By
Proposition~\ref{prop:deletions} there is a sequence of topological
hyperdeletions 
$\delta_1,\ldots,\delta_{|\alpha^{-1}\sigma|-1}$ taking
$(\sigma,\alpha)$ into $(\sigma,\theta)$. (Note that the number of
topological hyperdeletions must be $|\alpha^{-1}\sigma|-1$ since each
hyperdeletion 
decreases the number of faces by one.) The hypermap $(\sigma,\theta)$
has the same genus as $(\sigma,\alpha)$. Furthermore we have
\begin{align*}
2g(\eta,\theta^{-1}\sigma)
&=n+2-z(\eta)-z(\theta^{-1}\sigma)-z((\theta^{-1}\sigma)^{-1}\eta)
=n+2-z(\eta)-z(\zeta)-z(\zeta^{-1}\eta)\\
&=n+2-z(\eta)-z((\eta\zeta^{-1})^{-1}\zeta)-z(\zeta^{-1}\eta)=2g(\eta,\eta\zeta^{-1}),\\ 
\end{align*}
which, by our assumption implies
$g(\eta,\theta^{-1}\sigma)=g(\sigma,\alpha)$. Since $\eta\in
C_{\alpha^{-1}\sigma}(\alpha^{-1}\sigma,\sigma)$ also implies 
$g(\eta,\sigma)=0$, we obtain that $\eta$ belongs to
$C_{\theta^{-1}\sigma}(\theta^{-1}\sigma,\sigma)$. By
Proposition~\ref{prop:vertextour} there is a sequence of contractions
taking $(\sigma,\theta)$ into $(\eta,\eta\zeta^{-1})$.  
\end{proof}  

A compatible pair of tours $(\eta,\zeta)$ of a hypermap
$(\sigma,\alpha)$ allows to create a {\em two-disk diagram
  $D(\sigma,\alpha,\eta,\zeta)$} of  the hypermap $(\sigma,\alpha)$ on
an appropriate surface of genus $g(\sigma,\alpha)$ as follows:
\begin{enumerate}
\item We create a diagram of the hypermap $(\eta,\eta\zeta^{-1})$ with
  noncrossing lines on a surface of genus
  $g(\eta,\eta\zeta^{-1})=g(\sigma,\alpha)$ as follows. We take two disjoint disks and we list the points on the
  boundaries of both, in the cyclic order of $\eta$, respectively
  $\zeta$. By the result of Jacques~\cite{Jacques} there is a surface of
  genus $g(\sigma,\alpha)$ on which the disks representing $\eta$ and
  $\zeta$ may be drawn in such a way that filling in the hyperedges
  of $\eta\zeta^{-1}$ results in no crossings. See also
  Remark~\ref{rem:finite} below. 
\item We draw the vertices of $(\sigma,\alpha)$ as a noncrossing
  partition inside the disk bounded by the vertex tour $\eta$.
\item We draw the faces of $(\sigma,\alpha)$ as a noncrossing partition
  inside the disk bounded by the face tour $\zeta$.     
\end{enumerate}  

An illustration of the process is shown in
Figure~\ref{fig:bernardi2}. The diagram is drawn on a torus, the dashed
circle in the middle is the vertex tour $\eta$. The circle representing
the face tour $\zeta$ appears as the union of four slanted dashed line
segments in the picture, due to the usual toric identifications on the
boundaries of the picture. The shaded regions in the interior of the
disk bounded by $\eta$ are the vertices, that is, the cycles of
$\sigma$. The shaded regions in the interior of the
disk bounded by $\zeta$  (appearing as the four corners in the picture)
are the faces, that is, the cycles of $\alpha^{-1}\sigma$. Note that the
unshaded regions inside the disk bounded by $\eta$ may be labeled by the
hypercontractions and the unshaded regions inside the disk bounded by
$\zeta$ may be labeled by the topological hyperdeletions.

\begin{remark}
\label{rem:finite}
The second and the third steps of the above process are obvious, the
reader might wonder why not just apply the result of
Jacques~\cite{Jacques} directly. By creating a
two-disk diagram we reduce the  challenge of drawing an arbitrary
hypermap to that of drawing a unicellular monopole $(\eta,\eta\zeta^{-1})$. The surface to be used does not change if we add or remove points
that are fixed points of the collection of hyperedges
$\eta\zeta^{-1}$. For the map shown in Figure~\ref{fig:bernardi2},
removing all the fixed points of $\eta\zeta^{-1}$ results in the map
$((3,6,9,12),(3,9)(6,12))$ on the set of points $\{3,6,9,12\}$, which,
after renumbering, is isomorphic to the map $((1,2,3,4),(1,3)(2,4))$. 
Hence, to classify the surfaces we need, we may
restrict our attention to unicellular monopoles $(\eta,\eta\zeta^{-1})$
satisfying not only $z(\eta)=z(\zeta)=1$ but also that $\eta\zeta^{-1}$ has no
fixed point, implying $1\leq z(\eta\zeta^{-1})\leq n/2$. A direct
consequence of~(\ref{eq:genusdef}) to such a hypermap is
$2g(\eta,\eta\zeta^{-1})+1\leq n\leq 4
g(\eta,\eta\zeta^{-1})$
for any positive genus. Thus the number of substantially different
two-disk diagrams of genus $g>0$ is finite: an infinite number of
possibilities arises only after adding some fixed points to
$\eta\zeta^{-1}$ and selecting a pair of noncrossing 
partitions inside the two disks.
For $g=1$ we get $3\leq n\leq 4$. Without loss of generality we may assume
$\eta=(1,2,\ldots,n)$, then the only two genus $1$ unicellular monopoles
$(\eta,\eta\zeta^{-1})$ to consider are $((1,2,3),(1,3,2))$ and
$((1,2,3,4),(1,3)(2,4))$. In particular, we may create the two-disk
diagram of any genus $1$ map by drawing $((1,2,3,4),(1,3)(2,4))$ to a
torus and then refining the picture by adding more points and a pair of
noncrossing partitions to the picture. 
\end{remark}

We conclude this section with a two-disk diagram of the hypermap
presented in Example~\ref{e:robert}. This time we consider
$\gamma=(1,2)(5,6)$ and $\theta=(1,2)(5,6)(7,8,9)$, yielding
$$\eta=\gamma\sigma=(1,4,7,2,6,9,3,5,8)\quad\mbox{and}\quad \zeta=\theta^{-1}\sigma=(1,4,7,9,3,5,7,2,6,8).$$
We have $\eta\zeta^{-1}=(7,8,9)$ hence the drawing process starts with
creating a diagram for
$(\eta,\eta\zeta^{-1})=((1,4,7,2,6,9,3,5,8),(7,8,9))$. After the removal
of the fixed points of $\eta\zeta^{-1}$ we need to find a two-disk
diagram of $((7,9,8),(7,8,9))$ on the set of points
$\{7,8,9\}$. Essentially this is the only example of a unicellular
hypermonopole of genus $1$ that is not a map. 

\begin{figure}[h]
\begin{center}
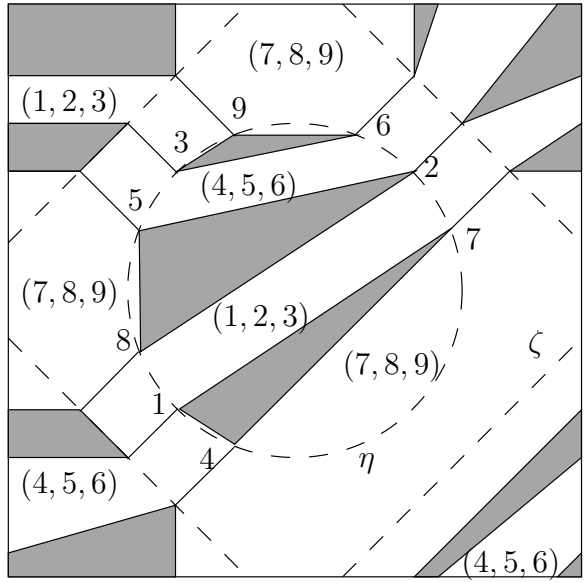
\end{center}
\caption{A $2$-disk diagram of the hypermap presented in Example~\ref{e:robert}.}
\label{fig:robert2} 
\end{figure}

Just like in the previous example, the unshaded regions inside $\zeta$
and $\eta$, respectively, represent the hyperdeletions and
hypercontractions, respectively. Note furthermore that, after removing
the dashed curves representing $\zeta$ and $\eta$, the unshaded regions
represent the hyperedges. For example there is a contiguous region labeled
$(1,2,3)$ and its part inside $\eta$ corresponds to the hypercontraction
$(1,2)$. Another contiguous region is labeled $(4,5,6)$ and its part
inside $\zeta$ corresponds to the hyperdeletion $(4,6)$. In general, when we
write $(\eta,\eta\zeta^{-1})=(\gamma\sigma,\gamma\alpha\delta)$, the
cycles of $\gamma$ label the unshaded regions inside $\eta$ and the
cycles of $\delta$ label the unshaded regions inside $\zeta$.

\section{Deletion-contraction processes}
\label{sec:dcp}

In this section we generalize the definition of deletion-contraction
processes used to define the Tutte polynomial, from 
connected graphs to hypermaps. Our generalization is inspired by the following
definition of the Tutte polynomial:
\begin{enumerate}
\item We fix a numbering of the edges of the graph, and we consider all
  deletion-contraction processes in which we delete or contract each
  edge in decreasing order of their numbers.
\item We may freely choose an edge to be deleted or contracted, except
  for the following two restrictions:
\begin{enumerate}
\item  We must contract an edge if deleting it would
  disconnect a graph. Such edges are {\em internally active}.
\item We must delete an edge if it is a loop (hence
  contracting it would not be topological). Such edges are {\em externally active}. 
\end{enumerate}
\item For each such deletion-contraction process we take the variable $x$ to
  the power of the number of internally active edges and multiply it
  with the variable $y$ to the power of the number of externally active
  edges. The Tutte polynomial is the sum of the contributions of all
  deletion-contraction processes.  
\end{enumerate}
It has been shown by Tutte that the resulting polynomial is independent
of the numbering of the edges. Furthermore, the set of (topologically)
contracted edges must be a spanning tree of the graph, each spanning
tree, together with the fixed numbering uniquely defines the
deletion-contraction process: all edges which were not contracted must
be deleted. 

This definition of a Tutte polynomial is immediately applicable to maps
as they are connected graphs with some additional topological structure,
which we may ignore. The only ``topological'' part of the definition of
the Tutte polynomial is that we disallow non-topological
hypercontractions which may be defined in abstract terms by disallowing
the contractions of loops. Furthermore, the final monopole with no edges
{\em depends on the set of contracted edges only}, not the order in which the
deletions and contractions are performed. The order matters only in the
definition of the activities. 

A plausible generalization of the above definition of a
deletion-contraction process to a hypermap 
$(\sigma,\alpha)$  is the 
following. We consider all sequences of (arbitrary) hyperdeletions and
topological hypercontractions which take hypermaps into hypermaps,
result in a hypermap $(\gamma\sigma,\gamma\alpha\delta)$ where
$z(\gamma\sigma)=1$ (at the end we have a single vertex) and
$\gamma\alpha\delta$ is the identity (``we have no hyperedges''). 

\begin{definition}
Given a hypermap $(\sigma,\alpha)$, a {\em deletion-contraction process}
is a sequence of hyperdeletions and topological hypercontractions
resulting in a hypermap $(\gamma\sigma,\gamma\alpha\delta)$ where
$z(\gamma\sigma)=1$  and $\gamma\alpha\delta$ is the identity. Each
hyperdeletion and hypercontraction is induced by a transposition, which
we call the {\em transposition underlying } the operation, and 
we call the set of all these transpositions the {\em set of
  underlying transpositions of  the deletion-contraction process}. 
\end{definition}  

\begin{definition}
Given a hypermap $(\sigma,\alpha)$ and a deletion-contraction process
for it, the {\em graph $G$ of underlying transpositions}  of the process
is the graph, whose vertices are the cycles of
$\sigma$ and whose edges are the underlying transpositions
$(u,v)$.
\end{definition}  

Note that the graph of underlying transpositions of a deletion-contraction
process depends only on the set of underlying transpositions of the process,
which set may be uniquely reconstructed from the graph. The graph $G$ is
not a map, as several underlying transpositions $(u,v)$ may involve the
same point $u$, but it is a {\em topological graph}, that is, a graph
whose edges may be cyclically ordered around each vertex
$(u_1,u_2,\ldots,u_k)$ as follows:
$$
\begin{array}{l}
(u_1,v_{1,1}),(u_1,v_{1,2}),\ldots,(u_1,v_{1,\ell_1}),\\
(u_2,v_{2,1}),(u_2,v_{2,2}),\ldots,(u_2,v_{2,\ell_2}), \ldots\\
(u_k,v_{k,1}),(u_k,v_{k,2}),\ldots,(u_1,v_{k,\ell_k}). 
\end{array}
$$
Here, for each $i\in\{1,2,\ldots,k\}$ the cycle
$(u_i,v_{i,1},v_{i,2},\ldots,v_{i,\ell_i})$ is the cycle of
$\alpha^{-1}$ containing the point $u_i$. Note that when we draw a hypermap
$(\sigma,\alpha)$ of genus zero following the conventions of
Figure~\ref{fig:hypermap}, we obtain a counterclockwise list of all
edges incident to the vertex $(u_1,u_2,\ldots,u_k)$.
Next we state some necessary conditions.

\begin{proposition}
\label{prop:sut}
Given a hypermap $(\sigma,\alpha)$, the set of underlying transpositions
of a deletion-contraction process must satisfy the following criteria:
\begin{enumerate}
\item For each transposition $(u,v)$ the points $u$ and $v$ belong to
  the same cycle of $\alpha$.
\item For each cycle $\alpha_i$ of $\alpha$, the transpositions swapping
  two points permuted by $\alpha_i$ form a tree on the set of points permuted
  by $\alpha_i$. This tree has non-crossing edges if we draw the points
  of $\alpha_i$ in cyclic order on a cycle.  
\end{enumerate}
\end{proposition}
\begin{proof}
Each transposition underlying a hyperdeletion or hypercontraction must
disconnect the current permutation representing the hyperedges, hence
each transposition $(u,v)$ must have both $u$ and $v$ in the same cycle
of the original permutation $\alpha$. Since no underlying transposition
is allowed to reconnect already disconnected cycles, the set of
transpositions swapping two points permuted by $\alpha_i$ must form the
edges of a cycle-free graph. Since, by the end of the process, the cycle
$\alpha_i$ must be completely disconnected, the set of underlying
transpositions swapping points permuted by $\alpha_i$ must form a tree on
all points permuted by $\alpha_i$. The non-crossing property may be shown
by induction on the number of points, keeping in mind that the first
transposition applied disconnects $\alpha_i$ into two cycles of
cyclically consecutive points. 
\end{proof}

\begin{definition}
\label{def:ltl}  
Given a hypermap $(\sigma,\alpha)$, we call the topological graph $G$ on
the vertex set of the cycles of $\sigma$, and indexed by ordered pairs of
points of the hypermap, {\em locally tree-like} if the set of
transpositions constituting the edge set of $G$ satisfy the criteria stated in
Proposition~\ref{prop:sut}. 
\end{definition}

The next necessary condition is on the subsets of underlying transpositions
that may represent the hypercontractions in a deletion-contraction
process.

\begin{proposition}
\label{prop:ctree}  
Let $(\sigma,\alpha)$ be a hypermap and let $G$ be the graph of
underlying transpositions for a deletion-contraction process. Consider
the subgraph $T$ of $G$ whose edges are the transpositions representing
hypercontractions. Then $T$ must be a spanning tree of $G$ and must
satisfy the following criterion: given any point $u$, and the cycle
$(u,v_1,v_2,\ldots,v_k)$ of $\alpha^{-1}$ containing $u$, the transpositions
representing hypercontractions on the list $(u,v_1),(u,v_2),\ldots,(u,v_k)$ must
precede the transpositions representing hyperdeletions. 
\end{proposition}
\begin{proof}
The subgraph $T$ has to be a spanning tree to assure that
the hypermap $(\gamma\sigma,\gamma\alpha\delta)$ obtained at the end of
a vertex has a single vertex and that all hypercontractions performed
are topological. The second criterion is a direct consequence of a result of
Goulden and Yong~\cite[Theorem 2.2]{Goulden-Yong}. This states that when 
we write a cyclic permutation as a minimal composition of
transpositions, then these transpositions must form the edges of a
noncrossing tree on the points represented on a cycle in a clockwise
order, and around each point the transpositions moving that point must
be performed in a counterclockwise order. Conversely, any sequence of
transpositions satisfying the above conditions is a factoring of the
cyclic permutation. Introducing $\gamma_i$
respectively $\delta_i$ to denote the product of transpositions
underlying the hypercontractions, respectively hyperdeletions acting on
the set of points permuted by the cycle $\alpha_i$ of $\alpha$, the permutation
$\gamma_i\alpha_i\delta_i$ must be the identity permutation,
equivalently we must have $\alpha_i^{-1}=\delta_i\gamma_i$. Note
that the transpositions constituting $\delta_i$ must be
all after the transpositions constituting $\gamma_i$.
\end{proof}

\begin{definition}
\label{def:dcm}  
Given a locally tree-like topological graph $G$ associated to a hypermap
$(\sigma,\alpha)$, we call a spanning tree $T$  {\em allowable} if it has
the property stated in Proposition~\ref{prop:ctree}. We say that
the topological graph is a
{\em deletion-contraction graph of $(\sigma,\alpha)$} if it is locally
tree-like and has an allowable spanning tree. We will use the notation
$(\sigma,\alpha)\models G$ to indicate that $G$ is a
deletion-contraction graph for $(\sigma,\alpha)$. 
\end{definition}

Definition~\ref{def:dcm} is justified by the next theorem, stating  that
the criteria listed in Propositions~\ref{prop:sut} and \ref{prop:ctree}
are also sufficient.

\begin{theorem}
\label{thm:mut}  
Given a hypermap $(\sigma,\alpha)$, a topological graph $G$ is the
map of underlying transpositions of a deletion-contraction process if
and only if $(\sigma,\alpha)\models G$ holds. Furthermore, all
deletion-contraction processes with the same graph of underlying
transpositions are fully characterized by the following criteria:
\begin{itemize}
\item[(i)] The edges of $G$ underlying a hypercontraction form an 
 allowable spanning tree of $G$.
\item[(ii)] For each cycle $\alpha_i$ of $\alpha$, if we list the points
  permuted by $\alpha_i$ in clockwise order, at each point $u$ the
  hyperdeletions whose underlying transposition moves $u$ are performed in
  counterclockwise order and the hypercontractions
  whose underlying transposition moves $u$ are performed in
  clockwise order.
\end{itemize}  
\end{theorem}   
\begin{proof}
By Propositions~\ref{prop:sut} and \ref{prop:ctree},
$(\sigma,\alpha)\models G$ is a necessary 
condition. To prove its sufficiency, it suffices to prove the stated
characterization of deletion-contraction processes and then show that the
hyperdeletions and hypercontractions may be performed in such order that
the conditions (i) and (ii) are satisfied.

Condition (i) is not only a necessary but also a sufficient condition to
assure that the hypermap $(\gamma\sigma,\gamma\alpha\delta)$ obtained at
the end of a vertex has a single vertex and that all hypercontractions performed
are topological. Condition (ii) is a direct consequence of the already
cited result of Goulden and Yong~\cite[Theorem
  2.2]{Goulden-Yong}. As pointed out in the proof of
Proposition~\ref{prop:ctree}, we must have
$\alpha_i^{-1}=\delta_i\gamma_i$ for each cycle $\alpha_i$ of $\alpha$
and this equation holds if and only if the transpositions constituting
$\delta_i\gamma_i$ are performed in in the cyclic order of
$\alpha_i^{-1}$ around each point moved by $\alpha_i$. The order in
which we apply the hypercontractions is the opposite to the order in
which the corresponding transpositions are composed to obtain
$\gamma_i^{-1}$ whereas the hyperdeletions are applied in the same order
as the corresponding transpositions are composed to obtain
$\delta_i^{-1}$.

We are left to show that for every map $G$ satisfying
$(\sigma,\alpha)\models G$ there is a deletion-contraction process whose
underlying graph of transpositions is $G$. First we select any spanning
tree of $G$ then we number the cycles of $\alpha$ in some order:
$\alpha_1$, $\alpha_2$, \ldots, $\alpha_m$. We will first number all
edges of $G$ contained in $\alpha_1$, then the ones
contained in $\alpha_2$, and so on. Hence all edges of $G$ contained in
the same $\alpha_i$ will be labeled consecutively. Let us represent the
points permuted by $\alpha_i$ on a circle in clockwise order. We only need
to show that the edges of this tree can be numbered in such a way that
for each point $u$ the numbers of all edges containing $u$ increase in
clockwise order. We may then reverse the numbering on the transpositions
underlying to hypercontractions and use the resulting numbering to label
the edges of $G$ contained in the set of points permuted by
$\alpha_i$. Executing all hyperdeletions and hypercontractions in
decreasing order of the labels satisfies the conditions (i) and (ii). 

The conclusion of the proof is a direct consequence of
Lemma~\ref{lemma:nctrees} below.
\end{proof}

\begin{lemma}
\label{lemma:nctrees}  
Given a set of $m$ points on a circle, numbered in clockwise order, and
a noncrossing tree on this set of points one may number the edges of
this tree in such a way that for each point $u$ the edges incident to
$u$ are numbered on increasing clockwise order. 
\end{lemma}
\begin{proof}
We proceed by induction on the number of points.
Consider a leaf of the tree. By cyclic rotation of the numbering, if
necessary, we may assume that this leaf is $(i,m)$ and that the point
numbered $m$ is not contained in any other edge. Deleting this edge
from the tree results in a pair of trees: a tree $T_1$ on the set
$\{1,2,\ldots,i-1,i\}$ and a tree $T_2$ on the set
$(i,i+1,i+2,\ldots,m-1)$. Consider the factorization
$$
(1,2,\ldots,m)=(1,2,\ldots,i-1,i)(i,m)(i,i+1,i+2,\ldots,m-1)
$$
and apply the induction hypothesis to the cycles $(1,2,\ldots,i-1,i)$
and $(i,i+1,i+2,\ldots,m-1)$, and the trees $T_1$ and $T_2$. Number the
edges of $T_1$ first, then number the edge $(i,j)$ and then number the
edges of $T_2$.
\end{proof}  

Deletion-contraction graphs may be equivalently described in terms of
the permutations $\gamma$, obtained by composing all the
hypercontractions of a deletion-contraction process. The proof of this
statement is more easily presented by using the notion of the 
Kreweras dual of a hypermap. 

\begin{theorem}
\label{thm:mut0}  
Let $(\sigma,\alpha)$ be a hypermap of genus $g$ on the set of points
$\{1,2,\ldots,n\}$. Let $G$ be a graph whose vertices are the cycles of
$\sigma$ and whose edges are transpositions $(u,v)$ where $1\leq u<v\leq
n$. Then  $G$ is a deletion-contraction graph, if and only if there is a
spanning genus $g$ unicellular hypermap $(\sigma,\gamma^{-1})$ of
$(\sigma,\alpha)$ such that the following are satisfied by the
permutations $\gamma$ and $\delta=\alpha^{-1}\gamma^{-1}$:
\begin{enumerate} 
\item Each edge of
$G$ either connects two points on the same cycle of $\gamma$, or on the
  same cycle of $\delta$.
\item The restriction of $G$ to any cycle of $\gamma$ or $\delta$ is a
  noncrossing tree if we represent the points in the cyclic order of the
  cycle of $\alpha$ containing the cycle of $\gamma$ or $\delta$.
\end{enumerate}  
\end{theorem}
\begin{proof}
Consider first a deletion-contraction process. The hyperdeletions of the
process leave the vertices unchanged, each topological contraction merges
two vertices (and refines a current hyperedge into two). The resulting monopole
$\gamma\sigma$ is a vertex tour of the hypermap, and conversely, every
vertex tour arises as the composition of topological
hypercontractions. As seen in Section~\ref{sec:hyperdc}, the
permutation $\gamma\sigma$ is a vertex tour of the hypermap if and
only if $(\sigma,\gamma^{-1})$ is a spanning hypertree of
$(\sigma,\alpha)$. Whenever this condition is satisfied, $\gamma\alpha$
is a refinement of $\alpha$, and the transpositions underlying a
hypercontraction must form a tree on each cycle of $\gamma$. Introducing
$\delta$ as the composition of all transpositions underlying a
hyperdeletion, the permutation $\gamma\alpha\delta$ is the identity
permutation if and only if $\delta=\alpha^{-1}\gamma^{-1}$ holds. The
transpositions underlying the hyperdeletions must form trees on
each cycle of $\delta$ since, starting with $\gamma\alpha$, each
hyperdeletion refines a hyperedge into two, and at the end of the
process $\gamma\alpha\delta$ is the identity permutation. 

To prove the converse, assume that $\alpha_1,\alpha_2,\ldots,\alpha_m$
are the cycles of $\alpha$ and for each $i$ let us denote by $\gamma_i$,
respectively $\delta_i$ the products of cycles of $\gamma$, respectively
$\delta_i$ contained in $\alpha_i$. To state that $\gamma\alpha$ is a refinement of $\alpha$ is
equivalent to stating that $(\alpha_i,\gamma_{i}^{-1})$ is a noncrossing
partition for each $i$. (Note that $\alpha_i$ is the single vertex
here!) Furthermore $\delta=\alpha^{-1}\gamma^{-1}$ is equivalent to
stating that $(\alpha_i,\delta_i^{-1})$ is the Kreweras dual of
$(\alpha_i,\gamma_{i}^{-1})$ for each $i$.  

\begin{figure}[h]
\begin{center}
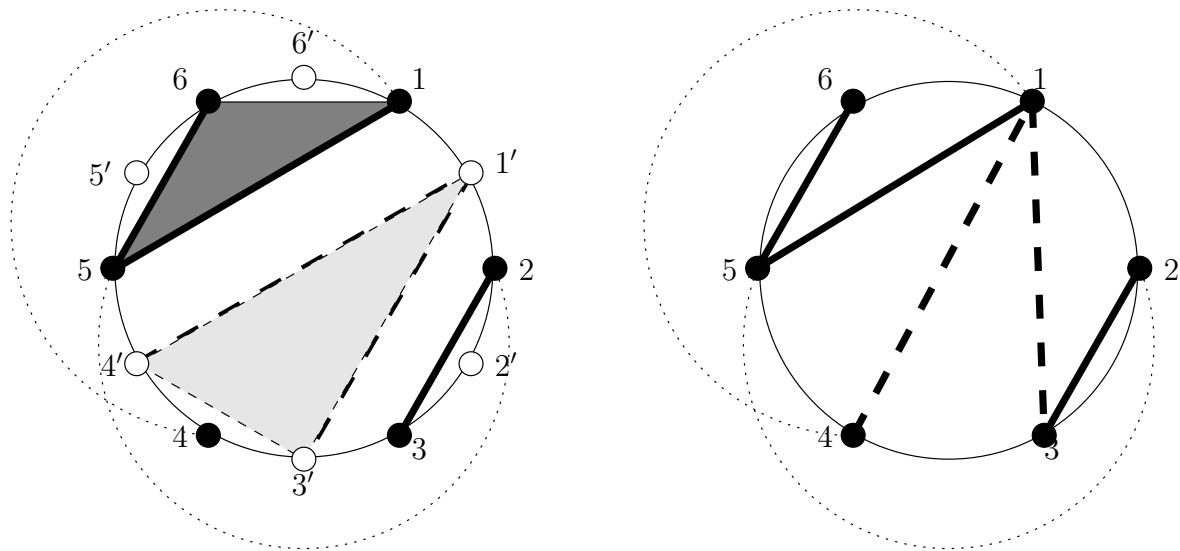
\end{center}
\caption{The hypermap $(\sigma,\alpha)$ with $\sigma=(1,4)(2,5)(3)(6)$,
  $\alpha=(1,2,3,4,5,6)$, $\gamma^{-1}=(1,5,6)(2,3)(4)$ and
  $\delta^{-1}=(1,3,4)(2)(5)(6)$.} 
\label{fig:Kdual} 
\end{figure}

At this point it is useful to recall the visualization of he Kreweras
dual, as introduced in~\cite{Kreweras}. The left hand side of
Figure~\ref{fig:Kdual} contains 
an illustration for a genus $1$ hypermap which, for simplicity's sake contains a
single hyperedge $\alpha_1=(1,2,3,4,5,6)$. We draw these on a circle in the
clockwise order and represent the cycles of $\gamma_1^{-1}$ on these
points. For each point $p$ we also introduce a new point $p'$ immediately
after $p$ in the clockwise order. The Kreweras dual of a noncrossing
partition is the coarsest noncrossing partition on the new points whose
parts do not intersect the parts of the original noncrossing
partition. For a general hypermap we repeat this representation for each
$\alpha_i$. Now we may take a tree on each cycle of $\gamma_i$,
respectively $\delta_i$ thus represented, and conclude by referring to
Lemma~\ref{lemma:nctrees}. In Figure~\ref{fig:Kdual} the
selected hypercontractions are marked with bold solid lines and the
hyperdeletions with bold dashed lines. The resulting admissible tree is
shown on the right hand side of Figure~\ref{fig:Kdual}.   
\end{proof}

The proof of Theorem~\ref{thm:mut0} contains the proof of the following,
stronger statement:
\begin{proposition}
\label{prop:can1}  
Let $(\sigma,\alpha)$ be a hypermap, let $G$ be a
deletion-contraction graph for it, and let $T$ be an admissible spanning
tree of $G$. Let $\gamma$ be the permutations obtained by the
transpositions marking the edges of $T$ composed in such an order that
each cycle of $\alpha$ transpositions incident to the same point are
performed in the cyclic order of $\alpha$. Let
$\delta=\alpha^{-1}\gamma^{-1}$. Then there is a deletion-contraction
process whose underlying graph of transpositions is $G$, the
transpositions underlying the hypercontractions are the edges of $T$,
and hypercontractions, respectively hyperdeletions belonging to the same
cycle of $\gamma$, respectively $\delta$ are performed consecutively.
\end{proposition}
Let us call the deletion-contraction processes described
Proposition~\ref{prop:can1} {\em  canonical}. 

Unfortunately, it seems unlikely that activities could be associated to
the hyperdeletions and hypercontractions in a way that their statistics
would be independent of the ordering of these operations. Consider the
example shown in Figure~\ref{fig:hypermap2}. This is the hypermap
$(\sigma,\alpha)$ with $\sigma=(1,4)(2,5)(3)$ and $\alpha=(1,2,3)(4,5)$.

\begin{figure}[h]
\begin{center}
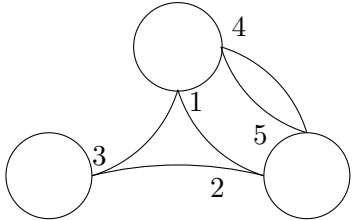
\end{center}
\caption{Hypermap of genus zero with two hyperedges}
\label{fig:hypermap2} 
\end{figure}

This hypermap has two hyperedges and three spanning hypertrees
$(\sigma,\gamma^{-1})$, the possible values of $\gamma^{-1}$ are
$(1,2,3)(4)(5)$, $(1)(2,3)(4,5)$ and $(1,3)(2)(4,5)$. In two of these
spanning hypertrees the edge $(4,5)$ is contracted, in one of them
it is deleted. If we try to replicate Tutte's proof swapping two
adjacent edge labels, and creating a bijection between spanning trees,
in the hypertree version we would need to match both hypertrees with
$(4,5)$ contracted to the only hypertree with $(4,5)$ deleted. There is
no such bijection. If we refine the picture to deletion-contraction
processes, the only spanning hypertree with a cycle of length $3$ gives
rise to three maps, whereas each of the other two
spanning hypertrees gives rise to only one map each. A Tutte style
bijection remains elusive. 

\section{Deletion-contraction formulas counting spanning hypertrees}
\label{sec:shtc}

Let $M=(\sigma, \alpha)$ be a map and $e=(i, j)$ one of its edges. It
is well known that the number of spanning trees of $M$ is equal to the
number of spanning trees of the map $M' = (\sigma, \alpha(i, j))$
obtained by deleting $e$ plus the number of spanning trees of the map $M''
= ((i, j)\sigma, (i, j)\alpha)$ obtained by contracting $e$.  It is
understood that deleting $e$ is not allowed if this disconnects the map
(in this case $M'$ is not a map, as $\sigma$ and $\alpha(i,j)$ do not
generate a transitive permutation group), and contracting $(i,j)$ is not
allowed if $(i,j)$ is a loop (and generates a non-topological
contraction). The justification is very simple: the spanning trees
of $M'$ are exactly the spanning trees of $M$ not containing the edge
$(i,j)$, and the spanning trees $((i,j)\sigma,\theta')$ of $M''$
correspond bijectively to those spanning trees $(\sigma,\theta)$ of
$M$ which contain the edge $(i,j)$, via the correspondence
$\theta'=(i,j)\theta$. 
 
The obvious generalization of this result does not hold for counting
spanning hypertrees of a hypermap, as the following example
shows. Consider the hypermap $(\sigma,\alpha)$ shown in
Figure~\ref{fig:hypermap2} and the transposition $(1,2)$. The hyperdeletion  
of $(1,2)$ gives the hypermap $(\sigma,\alpha(1,2))$ where
$\alpha(1,2)=(1,3) (2) (4,5)$. This hypermap has only one spanning hypertree
$(\sigma,(1,3)(4,5))$. The hypercontraction of $(1,2)$ gives the
hypermap $((1,2)\sigma,(1,2)\alpha)$ where $(1,2)\sigma=(1,4,2,5) (3)$
and $(1,2)\alpha=(1)(2,3) (4,5)$, which has the only spanning
hypertree $((1,2)\sigma,(2,3)(4)(5))$. However, as noted at the end
of the last section, the hypermap $(\sigma,\alpha)$ shown in
Figure~\ref{fig:hypermap2} has three hypertrees.

Hyperdeletions and hypercontractions still remain useful tools in
describing the set of spanning genus $g$ unicellular hypermaps of a
hypermap, because of the next two results.

\begin{proposition}
\label{prop:deletion}
Let $H = (\sigma, \alpha)$ be a hypermap and let $i,j$ be
two points belonging to the same cycle of $\alpha$. Let $g$ be any nonnegative
integer. If $H'=(\sigma,\alpha (i,j))$ obtained by applying the
hyperdeletion $(i,j)$, is a hypermap then its spanning genus $g$
unicellular hypermaps form a subset of the set of all spanning genus $g$ unicellular hypermaps of $H$: this
subset contains only spanning genus $g$ unicellular hypermaps
$(\sigma,\theta)$ for which $i$ and $j$ belong to different cycles of
$\theta$.   
\end{proposition}
\begin{proof}
Let  $(\sigma,\theta)$ be a spanning genus $g$ unicellular hypermap of
$H'$. Clearly $(\sigma, \theta)$ is a genus $g$ unicellular hypermap. In
order to check  that it is a spanning genus $g$ unicellular hypermap of
$H$ one has to prove that $\theta$ is a refinement of $\alpha$. This
holds  since $\theta$ is  a refinement of $\alpha(i,j)$  
and  $i,j $ being  in the same cycle of $\alpha$, the permutation
$\alpha(i,j)$ is a refinement of $\alpha$. 
\end{proof}
\begin{remark}
Regarding the interpretation of Proposition~\ref{prop:deletion} one
  should note that $H'=(\sigma,\alpha (i,j))$ may not be a hypermap,
  because the permutation group generated by $\sigma$ and $\alpha (i,j)$
  may not be transitive. This also happens when we delete an isthmus in a
  map. Furthermore $(i,j)$ may induce a non-topological hyperdeletion, which
  will decrease the genus. If $g$ is the genus of $H$ then none of the
  spanning genus $g$ unicellular hypermaps of $H$ will be a spanning
  genus $g$ unicellular hypermaps of a
  hypermap $H'$ obtained by a non-topological hyperdeletion. That said,
  the genus of a unicellular hypermap $(\sigma,\theta)$ does not change if
  we change the hypermap in which it is a spanning unicellular hypermap.  
\end{remark}  

Regarding hypercontractions we have the following analogous result.

\begin{proposition}
\label{prop:contraction}
Let $H = (\sigma, \alpha)$ be a hypermap and let $i,j$ be two points
belonging to the same cycle of $\alpha$. Assume that the
hypercontraction of $(i,j)$ in $H$ does not disconnect $H$: it results in
the hypermap $H''=((i,j)\sigma,(i,j)\alpha)$. Then the spanning
unicellular hypermaps of $H''$ are all hypermaps of the form
$((i,j)\sigma,(i,j)\theta)$ where $(\sigma,\theta)$ is any spanning
unicellular hypermap of $H$ satisfying that $i$ and $j$ belong to the
same cycle of $\theta$. Here
$g((i,j)\sigma,(i,j)\theta)=g(\sigma,\theta)$ if $i$ and $j$ belong to
different cycles of $\sigma$ and
$g((i,j)\sigma,(i,j)\theta)=g(\sigma,\theta)-1$ otherwise. 
\end{proposition}
\begin{proof}
Note first that the only situation when the hypercontraction of $(i,j)$
may disconnect $H$ is when $i$ and $j$ belong to the same cycle of
$\sigma$ and no spanning unicellular hypermap $(\sigma, \theta)$ of $H$
contains $i$ and $j$ on the same cycle of $\theta$. Indeed, only a
non-topological hypercontraction may disconnect hypermap. Assume by way
of contradiction that $i$ and $j$ belong to the same cycle of $\theta$
for some spanning unicellular hypermap $(\sigma, \theta)$. The
hypercontraction of $(i,j)$ takes $(\sigma, \theta)$ into $((i,j)\sigma,
(i,j)\theta)$ which is a unicellular hypermap since
$((i,j)\theta))^{-1}(i,j)\sigma=\theta^{-1}\sigma$ is a circular
permutation. Since $(i,j)$ disconnects $\theta$, the permutation
$(i,j)\theta$ refines $(i,j)\alpha$, hence $((i,j)\sigma, (i,j)\alpha)$
is also a hypermap. 

Let $((i,j)\sigma,\theta')$ be a genus $g$ unicellular hypermap of
$H''$. The points $i,j$ must belong to different cycles of $\theta'$ as
this permutation is a refinement of $(i,j)\alpha$ which contains $i$ and
$j$ on different cycles. Hence $\theta=(i,j)\theta'$ has $i$ and $j$ on
the same cycle.  If $i$ and $j$ belong to different
cycles of $\sigma$, then they  also belong to the same cycle of
$(i,j)\sigma$ and the map $(\sigma,\theta)\mapsto
((i,j)\sigma,(i,j)\alpha)$ is a topological hypercontraction, not
changing the number of faces, nor the genus. If $i$ and $j$ belong to
the same cycle of $\sigma$ and hence to different cycles of $(i,j)\sigma$ 
then the map $(\sigma,\theta)\mapsto
((i,j)\sigma,(i,j)\alpha)$ is a non-topological hypercontraction
decreasing the genus by one.  
\end{proof}

\begin{theorem}
\label{thm:stdecomp}
Let $H=(\sigma,\alpha)$ a hypermap such that $(1,2,\ldots,m)$ is a cycle
of $\alpha$. If $m\geq 2$ and $g\geq 0$ then the
set of all spanning genus $g$ unicellular hypermaps $(\sigma,\theta)$ of
$H$ is the disjoint union of the following sets $S_1,S_2,\ldots,S_m$:
\begin{itemize}
\item[(1)] $S_1$ is the set of all spanning genus $g$ unicellular hypermaps of
  $H_1=(\sigma,\alpha(1,m))$, obtained by the hyperdeletion of $(1,m)$
  in $H$. We set $S_1=\emptyset$ if the hyperdeletion of $(1,m)$
  disconnects $H$.
\item[(2)] Let $H_2=((1,2)\sigma,(1,2)\alpha)$ be the hypermap 
  obtained by applying the hypercontraction of $(1,2)$ to $H$. We set
  $S_2=\emptyset$ if the hypercontraction of $(1,2)$ disconnects
  $H$. Otherwise $S_2$ is the set of all spanning genus $g$ unicellular
  hypermaps of the form  
  $(\sigma,(1,2)\theta')$ where $((1,2)\sigma,\theta')$ is any spanning
genus $g$ (genus $g-1$) unicellular hypermap of $H_2$ if the
hypercontraction of $(1,2)$ is topological (non-topological).   
\item[(3)] For $k=3,\ldots,m$ we set
  $H_k=((1,k)\sigma,(1,k)\alpha(1,k-1))$, obtained by applying 
  the hyperdeletion of $(1,k-1)$ and the hypercontraction $(1,k)$ in
  $H$. We set $S_k=\emptyset$ if the hyperdeletion of $(1,k-1)$ and the
  hypercontraction $(1,k)$ disconnect $H$. Otherwise $S_k$ is the set of all
  genus $g$ unicellular hypermaps $(\sigma,(1,k)\theta')$, where
  $((1,k)\sigma,\theta')$ is any spanning genus $g$ (genus $g-1$)
  unicellular hypermap of the hypermap $H_k$ if the hypercontraction of
  $(1,k)$ is topological (non-topological).   
\end{itemize}  
\end{theorem}  
\begin{proof}
Given a spanning genus $g$ unicellular hypermap $(\sigma,\theta)$ of
$H$ we define 
$\phi(\theta)=1$ if $1$ is a fixed point of $\theta$ and as the second
smallest element of the cycle of $\theta$ containing $1$ otherwise.
We show for all $k$ that  $(\sigma,\theta)$ belongs to the set $S_k$ if and
only if $\phi(\theta)=k$. For this purpose let us describe the
hyperedges of the hypermaps $H_1,H_2,\ldots,H_m$. Since all
hyperdeletions and hypercontractions involve points on the cycle
$(1,2,\ldots,m)$ of $\alpha$, only hyperedges contained in the set
$\{1,2,\ldots,m\}$ are different in the hypermaps obtained by our
hyperdeletions and hypercontractions.

For $k=1$, we have
$$
(1,2,\ldots,m)(1,m)=(1)(2,3,\ldots,m).
$$
Any element of $S_1$ must satisfy $\phi(\theta)=1$. Conversely, if $1$
is a fixed point of $\theta$ then the spanning genus $g$ unicellular hypermap
$(\sigma,\theta)$ is also a spanning genus $g$ unicellular hypermap of $H_1$ by
Proposition~\ref{prop:deletion}. 

By Proposition~\ref{prop:contraction}, the spanning genus $g$
unicellular hypermap $(\sigma,\theta)$ belongs to $S_2$ if and only if
$1$ and $2$ belong to the same cycle of $\theta$, which is equivalent to
$\phi(\theta)=2$. 

For $k\geq 3$ we have
$$
(1,k)(1,2,\ldots,m)(1,k-1)=(1)(k,k+1,\ldots,m)(2,3,\ldots,k-1).
$$
Any spanning unicellular hypermap $((1,k)\sigma,\theta')$ of
$H_k$ must 
satisfy that $\theta'$ is a refinement of the above permutation, hence
the cycles of $\theta'$ containing the elements $2,3,\ldots,k-1$ can not
contain $1$ or $k$. The permutation $\theta=(1,k)\theta'$ has then $1$ and $k$
on the same cycle, but the elements $2,3,\ldots,k-1$ are still not on
this cycle, forcing $\phi(\theta)=k$. Conversely, if we have
$\phi(\theta)=k$ then $\theta$ is a refinement of 
$$
(1,2,\ldots,m)(1,k-1)=(1,k,k+1,\ldots,m)(2,3,\ldots,k-1)
$$
hence $(\sigma,\theta)$ is a spanning unicellular hypermap of
$(\sigma,\theta(1,k-1))$ by Proposition~\ref{prop:deletion} and $1$ and
$k$ are on the same cycle of $\theta$, hence
Proposition~\ref{prop:contraction} is applicable.  Note finally that for
$k\geq 2$ the stated relations between $g(\sigma,\theta)$ and
$g((1,k)\sigma,\theta')$ are direct consequences of
Proposition~\ref{prop:contraction}. 
\end{proof}  
\begin{corollary}
  If $H$ has at least one nontrivial hyperedge, we may use
  Theorem~\ref{thm:stdecomp} to write the number of its spanning
  genus $g$ unicellular hypermaps as the sum of the numbers of spanning
  genus $g$ and genus $g-1$ unicellular hypermaps of at most $m$ smaller
  hypermaps. Here $m$ is the size of the smallest nontrivial hyperedge in $H$.  
\end{corollary}  

\begin{remark}
We may use Theorem~\ref{thm:stdecomp} to find the number of spanning
hypertrees of the hypermap $H$ presented in Example~\ref{e:robert}. We
may apply the theorem to the cycles $\alpha_1=(1,2,3)$, 
$\alpha_2=(4,5,6)$ and $\alpha_3=(7,8,9)$ independently: let us
denote by $H_{i_1,i_2,i_3}$ the hypermap obtained by considering the
case $i_j$ for the cycle $\alpha_j$. Note that $H_{i_1,i_2,i_3}$ is
isomorphic to $H_{i_2,i_3,i_1}$ since the operation $x\mapsto x+3$ (modulo
$9$) on the set of points is an automorphism of $H$. If the list
$(i_1,i_2,i_3)$ contains more than one copy of $2$ or $3$ then
$H_{i_1,i_2,i_3}$ is obtained by using at least one non-topological
contraction, and may be disregarded. $H_{1,1,1}$ is not a hypermap: it is
disconnected. The hypermaps $H_{1,1,2}$, $H_{1,2,1}$ and $H_{2,1,1}$ are
maps consisting of $3$ parallel edges between two vertices, whereas $H_{1,1,3}$,
$H_{1,3,1}$ and $H_{3,1,1}$ are maps consisting of $2$ parallel edges
between two vertices. Finally, all $6$ hypermaps of the form
$H_{i_1,i_2,i_3}$ with no repeated index are hypermonopoles. We obtain that
the number of spanning hypertrees of $H$ is $3\times 3+3\times 2+6\times 1=21$.
\end{remark}

\begin{remark}
The refinements of the cycle $(1,2,\ldots,m)$, ordered by the
refinement relation, form a poset that is isomorphic to the lattice
of noncrossing partitions on the set $\{1,2,\ldots,m\}$. Since $\alpha$
contains the cycle $(1,2,\ldots,m)$, every spanning genus $g$
unicellular hypermap $(\sigma,\theta)$ has the property that each cycle
of $\theta$ acts on a set that is either contained in or disjoint from the set
$\{1,2,\ldots,m\}$. The answer to the question whether $(\sigma,\theta)$
belongs to $S_k$ only depends on the cycles of $\theta$ acting on the
set $\{1,2,\ldots,m\}$. Remarkably $(\sigma,\theta)$ belongs to $S_k$ if and
only if the noncrossing partition corresponding to the cycles of $\theta$
acting on the set $\{1,2,\ldots,m\}$ belongs to the set $R_k$ defined by
Simion and Ullman~\cite[Theorem 2]{Simion-Ullman} as an aid to
recursively construct a symmetric chain decomposition of the noncrossing
partition lattice.  
\end{remark}

\begin{remark}
For spanning hypertrees the classes $S_1,S_2,\ldots,S_m$ may be visualized as
follows. Let us 
represent the cycle $(1,2,\ldots,m)$ of $\alpha$ in this cyclic order on
the boundary of a disc, and let us also represent the vertex tour of the
spanning hypertree with a dotted curve as shown in
Figure~\ref{fig:turns}. The class $S_1$ contains those spanning
hypertrees whose vertex tour cuts out the point $1$ from the
hyperedge, leaving the points $2,3,\ldots,m$ on the outside. For $k\geq 2$ the
class $S_k$ contains exactly those spanning hypertrees whose vertex
tour contains the points $1$ and $k$ inside the tour (that is, on the
side containing the spanning hypertree) and leaves the points
$2,3,\ldots,k-1$ of the hyperedge on the outside. Nothing can be assumed
regarding the points $k+1,k+2,\ldots, m$. In a way, the index $k$ of the
set $S_k$ determines ``which way we turn'' when the vertex tour arrives near
the point $1$.
\end{remark}
  
\begin{figure}[h]
\begin{center}
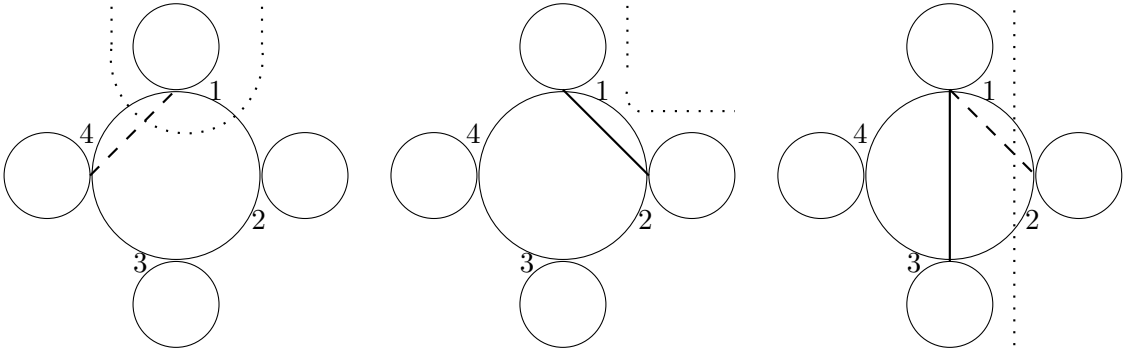
\end{center}
\caption{A local image of the vertex tour for the classes $S_1$, $S_2$,
  and $S_3$ in the case when $m=4$.}
\label{fig:turns} 
\end{figure}

\section{Semimeanders, meanders and reciprocals of monopoles and bipoles}
\label{sec:semimeanders}

\begin{definition}
For each integer $m\geq 0$ we define the {\em monopole with $m$ nested
  edges} as the hypermap, whose vertex permutation is
$(2,4,\ldots,2m,2m-1,2m-3,\ldots,3,1)$ and whose edge permutation is the
involution $(1,2)(3,4)\ldots (2m-1,2m)$. For each $m\in \{\frac{1}{2},
\frac{3}{2}, \frac{5}{2},\ldots\}$ we
define the {\em monopole with $m$ nested 
  edges} as the hypermap, whose vertex permutation is
$(2,4,\ldots,2m-1,2m,2m-2,\ldots,3,1)$ and whose edge permutation is the
involution $(1,2)(3,4)\ldots (2m-2,2m-1)(2m)$.  
\end{definition}
The reciprocal of a monopole with $2.5$, respectively $3$, nested edges
may be seen in Figure~\ref{fig:semimeander}, respectively
\ref{fig:semimeander2} below. Note that we may think of the loop
$(3)$ in Figure~\ref{fig:semimeander} as a ``half of an edge''.    

\begin{theorem}
The number of semimeanders of order $n$ equals the number of spanning
hypertrees of the reciprocal of a monopole with $n/2$ nested edges.
\label{thm:semimeanders}
\end{theorem}  
\begin{proof}
We prove the theorem by showing that there is a bijection between
the set of spanning hypertrees of the reciprocal of a monopole
with $n/2$ nested edges and the set of foldings of $n-1$ stamps defined
in Section~\ref{sec:pmeanders}. Introducing $m=\lfloor n/2\rfloor$, we
  have $n=2m+1$ if $n$ is odd and $n=2m$ if $n$ is even.

If $n$ is odd, then the reciprocal of the
monopole with $n/2$ nested edges is $(\sigma,\alpha)$ where
$$
\sigma=(1,2)(3,4)\cdots(2m-1,2m)(2m+1)\quad\mbox{and} 
$$
$$\alpha=(2,4,\ldots,2m,2m+1,2m-1,2m-3,\ldots,3,1).$$
Figure~\ref{fig:semimeander} below is an example of the case when
$n=5$. 
\begin{figure}[h]
\begin{center}
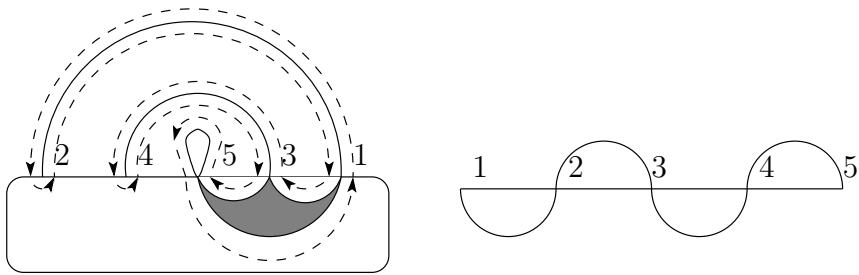
\end{center}
\caption{Spanning hypertree of a reciprocal monopole with $2.5$ nested
  edges and the corresponding stamp folding} 
\label{fig:semimeander} 
\end{figure}
Note that the faces of our hypermap are given by
$\alpha^{-1}\sigma=(1)(2,3)(4,5)\cdots (2m,2m+1)$. 
To any spanning hypertree $\theta$ of $(\sigma,\alpha)$ we associate the
one line representation described in Proposition~\ref{prop:1line0}. 
On the right hand side of Figure~\ref{fig:semimeander} we see the one
line representation associated to the spanning hypertree
$\theta=(1,5,3)(2)(4)$. This corresponds to
$\theta^{-1}\sigma=(1,2,3,4,5)$, we list the points on the line in this
order. Above the line we see the faces $\alpha^{-1}\sigma$ of our
hypermap, below the line we see the vertices $\sigma$. Note that the
union of the set of non-singleton cycles of $\alpha^{-1}\sigma$ and of
$\sigma$ are all two-cycles of the form $(i,i+1)$ for $\leq i\leq
n-1$. This union is a disjoint union, the two-cycles with an even
smaller element belong to $\alpha^{-1}\sigma$, the others belong to
$\sigma$. Hence we obtain the diagram of a stamp folding.  

A similar bijection may be constructed when $n$ is even. In this case we
have 
$$
\sigma=(1,2)(3,4)\cdots(2m-1,2m)\quad\mbox{and} 
$$
$$\alpha=(2,4,\ldots,2m,2m-1,2m-3,\ldots,3,1).$$
Figure~\ref{fig:semimeander2} below is an example of the case when
$n=6$. 
\begin{figure}[h]
\begin{center}
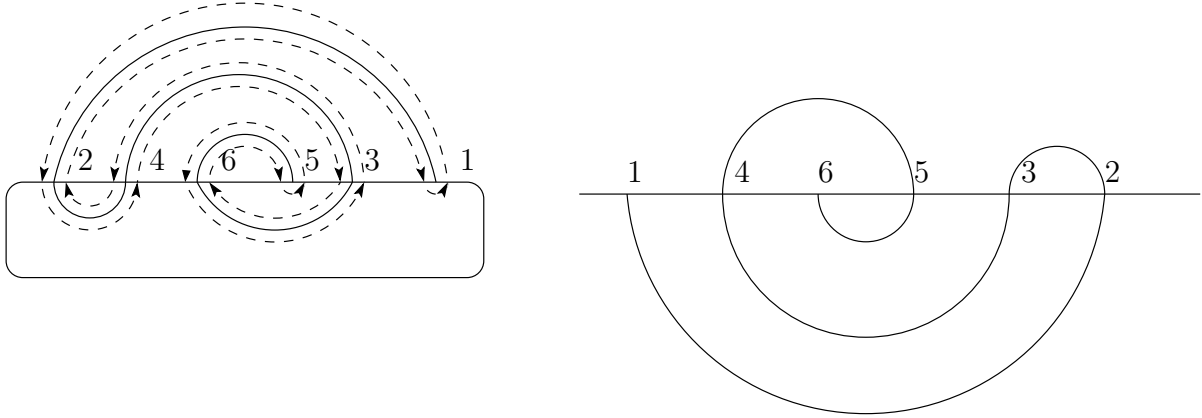
\end{center}
\caption{Spanning hypertree of a reciprocal monopole with $3$ nested
  edges and the corresponding stamp folding}
\label{fig:semimeander2} 
\end{figure}

In this case we have
$\alpha^{-1}\sigma=(1)(2,3)(4,5)\cdots (2m-1,2m-1)(2m)$.
The example in Figure~\ref{fig:semimeander2} represents the spanning
hypertree $\theta=(1)(2,4)(3,6)(5)$ and we have
$\theta^{-1}\sigma=(1,4,6,5,3,2)$. Once again the one line
representation described in Proposition~\ref{prop:1line0} is a stamp
folding diagram. 
\end{proof}

\begin{definition}
We define the {\em dipole with $n$ parallel edges} as the map whose
vertex permutation is $(1,3,5,2n-1)(2n,2n-2,\ldots,4,2)$ and whose edges
are $(2n,1)(2,3)(4,5)\cdots (2n-2,2n-1)$.
\end{definition}  

\begin{theorem}
The number of meanders of order $n$ equals the number of spanning
hypertrees of the reciprocal of a dipole with $n$ parallel edges.
\label{thm:meanders}
\end{theorem}
\begin{proof}
The proof is analogous to the proof of Theorem~\ref{thm:semimeanders}.
The reciprocal hypermap $\sigma\alpha$ is given by 
$$
\sigma=(2n,1)(2,3)(4,5)(2n-2,2n-1)\quad\mbox{and}\quad
\alpha=(1,3,\ldots,2n-1)(2n,2n-2,\ldots,4,2).  
$$
Note that we have $\alpha^{-1}\sigma=(1,2)(3,4)\ldots (2n-1,2n)$.

An example of the reciprocal of a dipole with $4$ parallel edges is
shown in Figure~\ref{fig:meander}. 
\begin{figure}[h]
\begin{center}
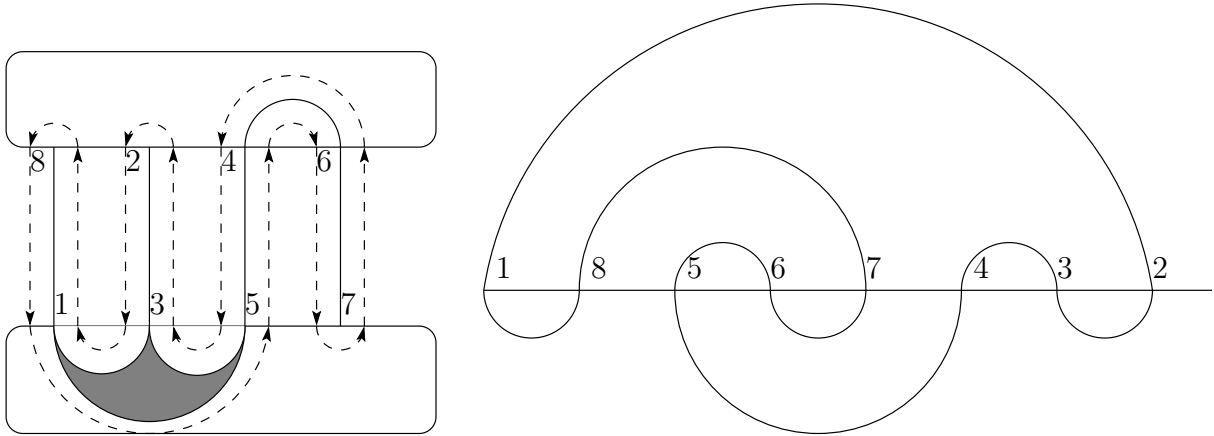
\end{center}
\caption{Spanning hypertree of a reciprocal dipole with $4$ parallel
  edges and the corresponding meander}
\label{fig:meander} 
\end{figure}

The example in the figure represents the spanning hypertree
$\theta=(1,3,5)(2)(4,6)(7)(8)$, yielding
$\theta^{-1}\sigma=(1,8,5,6,7,4,3,2)$.  
\end{proof} 
\begin{remark}
Since any spanning hypertree $\theta$ of $(\sigma,\alpha)$ is a
refinement of $\alpha$, a permutation with two cycles of the same size,
we may think of each spanning hypertree as a pair of noncrossing
partitions on the same $n$-element set. This identification is easier to
visualize if we relabel $2i+1$ as $i$ for $1\leq i\leq n-1$, relabel $1$
as $n$, and relabel $2i$ as $i'$ for $1\leq i\leq n$. We then obtain
$$
\sigma=(1,1')(2,2')(3,3')(n,n')\quad\mbox{and}\quad
\alpha=(1,2,\ldots,n)(n',(n-1)',\ldots,1').  
$$
Each cycle of $\theta$ contains either labels from the set
$\{1,\ldots,n\}$ or from the set $\{1',\ldots,n'\}$. The cycles of
$\theta$ contained in $\{1,\ldots,n\}$  must form a noncrossing
partition, represented in the usual way, and the cycles of $\theta$
contained in $\{1',\ldots,n'\}$ must from a noncrossing partition,
represented with cycles in the reverse order compared to the usual
way. Thus we arrive precisely at the representation first developed by
Franz~\cite{Franz-po}. The structure of this representation was further
studied and utilized in~\cite{Franz-representation} and in
\cite{Franz-Earnshaw}.   

\end{remark}

\section{Spanning hypertrees of reciprocals of maps}
\label{sec:rtrees}

In this section we generalize the construction of Franz~\cite{Franz-po}
defining labeled plane trees representing meanders to define
labeled plane trees representing spanning hypertrees of the
reciprocal of a map. Our construction is illustrated with the map and a spanning
hypertree of its reciprocal shown in Figure~\ref{fig:rtree}.

\begin{figure}[h]
\begin{center}
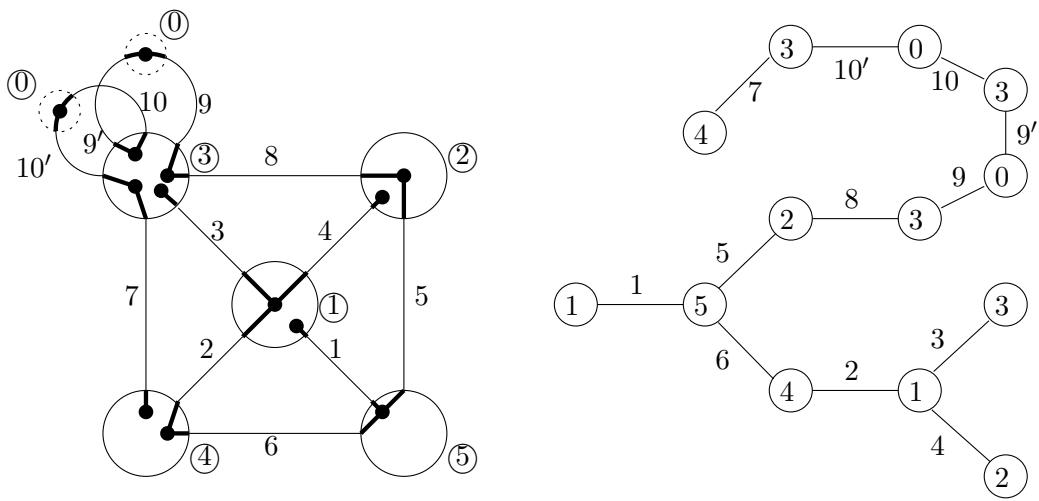
\end{center}
\caption{A map and a spanning hypertree of its reciprocal}
\label{fig:rtree} 
\end{figure}

Given a map $(\sigma,\alpha)$, we number its vertices (these numbers are
circled in Figure~\ref{fig:rtree}) and we number its
edges. Furthermore, for each loop edge (that is, a $2$-cycle of $\alpha$,
whose points belong to the same cycle of $\sigma$) whose number is $i$
we also associate a duplicate label $i'$. Thanks to this duplication,
every point of the map may be uniquely described by the ordered pair
$(i,j)$ where $i$ is the label of the vertex and $j$ is the label of the
edge containing it. For example, after identifying each point with its
pair of labels, cycle number $3$ of $\sigma$  is $((3,7),(3,3),(3,8),
(3,9), (3,10), (3,9'), (3,10'))$, cycle number $5$ of $\alpha$ is
$((2,5),(5,5))$ and cycle number $9$ of $\alpha$ is $((3,9),(3,9'))$.   
\begin{definition}
We call a numbering of vertices and edges of a map, together with the
labeling of its points as defined above a {\em vertex-edge labeling} of
the map.    
\end{definition}

Next we fix a spanning hypertree $(\alpha,\theta)$ of the reciprocal
hypermap $(\alpha,\sigma)$. The permutation $\theta$ is a refinement of
the permutation $\sigma$: for each cycle of $\sigma$, the cycles of
$\theta$ contained in it form a genus zero permutation. For example the
cycles of $\theta$ contained in the first cycle of $\sigma$ form the
permutation $((1,1)), ((1,2),(1,4),(1,3))$. We associate a {\em center} to
each cycle of $\theta$ and to each loop edge of $\alpha$: these are
small black disks in Figure~\ref{fig:rtree}. Note that the dotted
circles around the centers of the loop edges of $\alpha$ are {\em not}
vertices of our map $(\sigma, \alpha)$, but they will become vertices in
the associated 
tree on the right hand side. We connect each center to
the points on its cycle (of $\theta$ or $\alpha$, shown in bold). By
merging the edges connecting the centers 
to the points on their cycles with the edges of our map, we obtain a
graph on the centers. For example, in
Figure~\ref{fig:rtree} we connect the center of the cycle
$((1,2),(1,4),(1,3))$ to the point $(1,2)$, we continue this edge using
the cycle $((1,2),(4,2))$ of $\alpha$ to the point $(4,2)$ and then we
continue to the center of the cycle $((4,2)(4,6))$. The resulting
topological graph is shown on the right hand side of
Figure~\ref{fig:rtree}. We label the center of each loop edge of
$\alpha$ with $0$, and we label all other centers with the number of the
cycle of $\sigma$ containing the cycle of $\theta$ containing the
center. Note that the second coordinates of all points on an edge
connecting two centers are the same: we label the edge connecting two
centers with this coordinate. The map we see in Figure~\ref{fig:rtree} is
a  {\em coherently $(\sigma,\alpha)$-labeled plane tree}. We explain its
definition in several steps.

\begin{definition}
\label{def:sa-labeled map}
Given a map $(\sigma,\alpha)$ with a vertex-edge labeling, a {\em
  $(\sigma,\alpha)$-labeled map $(\sigma',\alpha')$} is a map with
numbered vertices and edges, subject to the following rules:
\begin{enumerate}
\item Each vertex of $(\sigma',\alpha')$ is either labeled with zero or
  it is labeled   with the first coordinate of a point in the
  vertex-edge labeling of $(\sigma,\alpha)$. The same label may appear
  on several vertices of $(\sigma',\alpha')$.  
\item The set of edge labels of $(\sigma',\alpha')$ is the set of the second
  coordinates of the points in the vertex-edge labeling
  $(\sigma,\alpha)$. Each edge label appears exactly once.
\item The vertices labeled $0$ of $(\sigma',\alpha')$ have degree $2$
  and they correspond to the set of loop edges of $(\sigma,\alpha)$
  bijectively: if $j$ is the label of a loop edge in $(\sigma,\alpha)$
  then there is exactly one vertex labeled $0$ of $(\sigma',\alpha')$
  that us incident to a pair of edges labeled $j$ and $j'$ respectively.
\end{enumerate}
\end{definition}  

Alternatively, using Definition~\ref{def:split}, we may describe
$(\sigma,\alpha)$-labeled maps as follows.

\begin{proposition}
The map $(\sigma',\alpha')$ is a $(\sigma,\alpha)$-labeled map if and
only if its 
  underlying graph may be obtained from the underlying graph of
  $(\sigma,\alpha)$ by subdividing each loop edge into two edges and
  then applying several vertex splitting operations (as defined in
  Definition~\ref{def:split}) that never split the newly introduced
  subdividing vertices.  
\end{proposition}  

Indeed, the vertices of $(\sigma',\alpha')$ that were added as vertices
subdividing the loop edges are labeled with zero, each other vertex
of $(\sigma',\alpha')$ is labeled with the number of the vertex of
$(\sigma,\alpha)$ that was split (possibly several times) to obtain it). The 
edges of $(\sigma',\alpha')$ are identifiable with the edges of
$(\sigma,\alpha)$ after subdividing each loop edge of $(\sigma,\alpha)$
into two edges. Each vertex of $(\sigma',\alpha')$ of color $i>0$ is a
cyclic permutation acting on a subset of points moved by cycle number
$i$ of $\sigma$ and the collection of all such cyclic permutations of
the same color $i>0$ is a permutation of all points moved by cycle
number $i$ of $\sigma$. 
\begin{definition}
\label{def:coherent}
  Given a map $(\sigma,\alpha)$ with a vertex-edge labeling and a
  $(\sigma,\alpha)$-labeled map  $(\sigma',\alpha')$, we call the
  $(\sigma,\alpha)$-labeling {\em coherent} if 
  for each vertex label $i>0$, the cycles forming the set of vertices of
  color $i$ in $\sigma'$ form a refinement of the unique cycle numbered
  $i$ in $\sigma$.    
\end{definition}

In the case when the $(\sigma,\alpha)$-labeled map $(\sigma',\alpha')$
is a tree (of genus $0$), we may draw it in the plane in such a way that
for each vertex of $\sigma'$ of positive color the incident edges listed
in the counterclockwise order mark the corresponding points in the order
of the cycle of $\sigma'$. We call the resulting plane tree
corresponding to a coherently $(\sigma,\alpha)$-labeled map a 
{\em coherently $(\sigma,\alpha)$-labeled plane tree}.

\begin{theorem}
\label{thm:rtrees}  
For a map $(\sigma,\alpha)$,  there is a bijection between its spanning
hypertrees and the coherently $(\sigma,\alpha)$-labeled plane trees. 
\end{theorem}  
\begin{proof}
Let the map $(\sigma',\alpha')$ be a coherently
$(\sigma,\alpha)$-labeled plane tree. Note that the edge set of
$\alpha'$ is obtained by subdividing each loop edge in $\alpha$ into two
edges, and the requirement of coherence is equivalent to requiring that
the restriction $\sigma''$ of $\sigma'$ to the set of the original
points  must be a refinement of $\sigma$. Conversely, for any
refinement $\sigma''$ of $\sigma$, we may introduce a center to each
cycle of $\sigma''$, a center to each loop edge of $(\sigma,\alpha)$ and
create a labeled topological graph using the procedure described at the
beginning of this section. We only need to show that $(\alpha,\sigma'')$
is a hypertree if and only if the map $(\sigma',\alpha')$ is a tree. The
hypermap $(\alpha,\sigma'')$ is a hypertree if and only if
$z(\sigma''^{-1}\alpha)=1$. The map
$(\sigma',\alpha')$ is a tree if and only if
$z(\alpha'^{-1}\sigma')=z(\sigma'^{-1}\alpha')=1$. The statement is now
a consequence of the fact that $\sigma'^{-1}\alpha'$
may be computed from $\sigma''^{-1}\alpha$ 
by replacing $(\ldots i \sigma''^{-1}(j)\ldots ) $  with $(\ldots i
p_{i,j} \sigma''^{-1}(j)\ldots )$ for each loop edge $(i,j)$ in $\alpha$. Here
$p_{i,j}$ is the midpoint inserted in $(i,j)$ to obtain
$\alpha'$.  
\end{proof}  

Theorem~\ref{thm:rtrees} and its justification have the following
consequence.

\begin{corollary}
\label{cor:rtrees}  
The spanning hypertrees of the reciprocal $(\alpha,\sigma)$ of a map
$(\sigma,\alpha)$ are the reciprocals of all trees obtained from
$(\sigma,\alpha)$ by a sequence of topological vertex splittings. 
\end{corollary}  

Corollary~\ref{cor:rtrees} inspires considering the generation of the
spanning hypertrees of the reciprocal of a map by a sequence of
topological vertex splittings. Using this approach, the key move
presented in the work of Franz and Earnshaw~\cite{Franz-Earnshaw} may be
generalized as follows. Let $(\sigma,\alpha)$ be a map and consider a
spanning hypertree of its reciprocal, represented as a coherently
$(\sigma,\alpha)$-labeled plane tree $(\sigma',\alpha')$. Take two
vertices of the same color, representing adjacent blocks of $\sigma'$
that can be merged and still have a noncrossing partition. For example,
we may merge the two vertices of color $1$ of the plane tree shown in
Figure~\ref{fig:rtree} in such a way that the edges around the only
vertex of color $1$ are listed $(1,4,3,2)$ in the counterclockwise
order. The resulting plane graph has a unique cycle. In our example,
this cycle is a triangle with vertices of color $1$, $4$, $5$, and edges
labeled $1$, $2$, and $6$. We may obtain
another plane tree representing a spanning hypertree by performing a
topological vertex splitting that breaks this cycle without
disconnecting the graph. For example we may replace the unique cycle
$((5,1), (5,6),(5,5))$ of color $5$ with the pair of cycles
$((5,1), (5,6))((5,5))$ or with the pair of cycles $((5,1))
((5,6),(5,5))$. (We cannot use $((5,1),(5,5)), ((5,6))$ as the resulting
plane graph would still contain a cycle and the edge $((5,1),(1,1))$
would be disconnected from the rest of the graph. This {\em reciprocal
  tree flipping} is analogous to replacing a spanning tree $T$ with the
spanning tree $T-\{f\}\cup \{e\}$ where $e$ is an edge external to $T$
and $f$ belongs to the unique cycle contained in $T\cup \{e\}$. Franz
and Earnshaw~\cite{Franz-Earnshaw} apply this idea to maps with $2$
vertices and $n$ parallel (non-intersecting) edges, and they call the
operation a {\em reduction} if it merges the first block of the first
vertex with the block containing the first point not in the first block
and then splits the first available block in some order that breaks the
cycle. It is not hard to see that the idea of this reduction map may be
generalized to the reciprocal of an arbitrary map. Using some ordering
on the points, the reduction map becomes well-defined, its inverse is
not unique but has been useful in the constructive enumeration of meanders.  
  
A remarkable consequence of Theorem~\ref{thm:rtrees} is that for
loopless maps with vertices of degree at most three the number of spanning
hypertrees of the reciprocal only depends on the underlying graph and not
on the cyclic order of the edges around the vertices.

\begin{corollary}
\label{cor:crtree}
Let $(\sigma,\alpha)$ be a map such that each cycle of $\sigma$ has
length at most $3$ and each cycle of $\alpha$ has length $2$, containing
points of two different cycles of $\sigma$.  Let $G=(V,E)$ be a graph
whose vertices are the cycles of $\sigma$ and whose edges are
bijectively labeled with the cycles of $\alpha$ in such a way that the
edge labeled $(i,j)$ connects the vertices containing the points $i$ and
$j$ respectively. Then the spanning hypertrees of the reciprocal
hypermap $(\alpha,\sigma)$ are in bijection with all labeled trees $T$
having the following properties:
\begin{enumerate}
\item The edges of $T$ are bijectively labeled with the edge set $E$. 
\item Each vertex of $T$ is labeled with a vertex of $G$.
\item The vertex labeling is a proper coloring: no two vertices of $T$
  of the same vertex label are adjacent.
\item The set of edge labels of all edges incident to some vertex
  labeled with $v\in V$ is the set of all edges incident to $v$ in $G$.   
\end{enumerate}
\end{corollary}  
Indeed, by Theorem~\ref{thm:rtrees}, the spanning hypertrees of
$(\alpha,\sigma)$ are in bijection with $(\sigma,\alpha)$-labeled plane
trees in which each vertex has degree at most $3$. The criteria stated
in Corollary~\ref{cor:crtree} match the definition of a
$(\sigma,\alpha)$-labeled plane tree, except there is no topological
restriction on the order of the edges around the vertices. These
restrictions have no meaning for vertices of degree one or two, neither
on the $(\alpha,\sigma)$ side nor on the plane tree side. For each
vertex $v\in V$ of degree $3$, there are two possible cyclic orientations of the
edges around the vertex. This orientation becomes irrelevant if there
is more than one vertex labeled $v$ in $T$. Otherwise exactly one cyclic 
orientation of the edges incident to the unique $v$-colored vertex in
$T$ satisfies the definition of the $(\sigma,\alpha)$-labeling. On
the other hand, for any tree in which the maximum degree is three, each
planar embedding of the tree may be uniquely described by choosing a
cyclic orientation around each vertex of degree three, and these choices
may be made independently.

\section*{Acknowledgments}
The authors are indebted to two anonymous referees for the very careful
reading of this manuscript and many valuable suggestions.
The second author wishes to express his heartfelt thanks to Labri, Universit\'e
Bordeaux I, for hosting him as a visiting researcher in Summer 2019,
when this research was started. This work was partially supported by a
grant from the Simons Foundation (\#514648 to G\'abor Hetyei).


\begin{thebibliography}{99}

\bibitem{Bernardi-embeddings}
O.\ Bernardi, 
A characterization of the Tutte polynomial via combinatorial embeddings,
{\it Ann.\ Comb.\ } {\bf 12} (2008), 139--153. 

\bibitem{Bernardi-sandpile}
O.\ Bernardi, 
Tutte polynomial, subgraphs, orientations and sandpile model: new
connections via embeddings, 
{\it Electron.\ J.\ Combin.\ } {\bf 15} (2008), Research Paper 109, 53 pp.   

\bibitem{Bernardi-Kalman-Postnikov}
O.\ Bernardi, T.\ Kalman and A.\ Postnikov,
Universal Tutte polynomial,
preprint 2020, arXiv:2004.00683 [math.CO]. 

\bibitem{Cori}
R.\ Cori,
``Un code pour les Graphes Planaires et ses applications,''
{\it Asterisque} {\bf 27} (1975).

\bibitem{Cori-hrec}
  R.\ Cori,
  Codage d'une carte planaire et hyperarbres recouvrants,
Colloques Internat.\ C.N.R.S, Orsay (1976).    
  
\bibitem{Cori-Machi}
  R.\ Cori and A.\ Mach\`\i,
  Su alcune propriet\`a del genere di una coppia di permutazioni,
{\it  Boll.\ Un.\ Mat.\ Ital. } (5) 18-A (1981), 84--89. 
  
\bibitem{Cori-Penaud}
R.\ Cori and J-G.\ Penaud, 
The complexity of a planar hypermap and that of its dual,
{\it Ann.\ Discrete Math.\ } {\bf 9} (1980), 53--62. 

\bibitem{Eades-dM}
P.\ Eadesand and C.\ F.\ X.\ de Mendon\c{c}a Neto, 
Vertex splitting and tension-free layout, Graph drawing (Passau, 1995),
202--211, Lecture Notes in Comput. Sci., 1027, Springer, Berlin, 1996. 

\bibitem{Flajolet-Sedgewick}
P.\  Flajolet and R.\ Sedgewick, 
Analytic combinatorics, Cambridge University Press, Cambridge, 2009.

\bibitem{DiFrancesco}
P.\ Di Francesco, O.\ Golinelli and E.\ Guitter, 
Meander, folding, and arch statistics, 
Combinatorics and physics (Marseilles, 1995),
{\it Math.\ Comput.\ Modelling} {\bf 26} (1997), 97--147. 

\bibitem{Franz-po}
  R.O.W.\ Franz,
  A partial order for the set of meanders,
  {\it Ann.\ Comb.\ } {\bf 2} (1998), 7--18.   
  
\bibitem{Franz-representation}
R.O.W.\ Franz, 
On the representation of meanders, 
{\it Bayreuth.\ Math.\ Schr.\ } (2003), 163--201. 

\bibitem{Franz-Earnshaw}
R.O.W.\ Franz and B.A.\ Earnshaw, 
A constructive enumeration of meanders,
{\it Ann.\ Comb.\ } {\bf 6} (2002), 7--17.   

\bibitem{Goulden-Yong}
I.\ Goulden and A.\ Yong,
Tree-like properties of cycle factorizations,
{\it J.\ Combin.\ Theory Ser.\ A} {\bf 98} (2002), 106--117.   

\bibitem{Jacques}
A.\ Jacques,
Sur le genre d'une paire de substitutions,
{\it C.\  R.\ Acad.\ Sci.\ Paris} {\bf 267} (1968), 625--627.

\bibitem{Kalman}
T.\ K\'alm\'an, 
A version of Tutte's polynomial for hypergraphs, 
{\it Adv.\ Math.\ } {\bf 244} (2013), 823--873.   

\bibitem{Kalman-Postnikov}
T.\ K\'alm\'an and A.\ Postnikov,
Root polytopes, Tutte polynomials, and a duality theorem for bipartite
graphs,
{\it Proc.\ Lond.\ Math.\ Soc.\ } (3) {\bf 114} (2017), 561--588. 

\bibitem{Kalman-Tothmeresz}
T.\ K\'alm\'an and L.\ T\'othm\'er\'esz, 
Hypergraph polynomials and the Bernardi process,
{\it Algebr.\ Comb.\ } {\bf 3} (2020), 1099--1139. 
  
\bibitem{Kreweras} 
G.\ Kreweras, 
Sur les partitions non crois\'ees d'un cycle,
{\it Discrete Math.} {\bf 1} (1972), 333--350.

\bibitem{LaCroix}
M.\ LaCroix, Approaches to the enumerative theory of meanders, Masters
Essay, University of Waterloo, Waterloo, Ontario, Canada, 2003.
\url{https://www.math.uwaterloo.ca/~malacroi/Latex/Meanders.pdf}. 


\bibitem{Legendre}
S.\ Legendre, 
Foldings and meanders,
{\it Australas.\ J.\ Combin.\ } {\bf 58} (2014), 275--291. 

\bibitem{Lunnon}
W.\ F.\ Lunnon, 
A map-folding problem,
{\it Math.\ Comp.\ } {\bf 22} (1968), 193--199.   

\bibitem{Machi}
A.\ Mach\`\i, 
On the complexity of a hypermap,
{\it Discrete Math.\ } {\bf 42} (1982), 221--226.   

\bibitem{OEIS}
N.J.A.\ Sloane,
``On-Line Encyclopedia of Integer Sequences,''\\
{\tt http://www.research.att.com/\~{}njas/sequences}

\bibitem{Rosenstiehl}
P.\ Rosenstiehl, 
Planar permutations defined by two intersecting Jordan curves,
Graph theory and combinatorics (Cambridge, 1983), 259--271, Academic
Press, London, 1984.    

\bibitem{Simion-Ullman}
R.\ Simion and D.\ Ullman, 
On the structure of the lattice of noncrossing partitions,
{\it Discrete Math.\ } {\bf 98} (1991), 193--206.   

\bibitem{Touchard}
J.\ Touchard,
Contribution \`a l'\'etude du probl\`eme des timbres poste,
{\it Canad.\ J.\ Math.\ } {\bf 2} (1950), 385--398.  

\bibitem{Tutte}
W.\ T.\ Tutte, 
Duality and trinity. Infinite and finite sets
(Colloq., Keszthely, 1973; dedicated to P. Erd\H os on his 60th
birthday), Vol. III, pp. 1459--472. Colloq. Math. Soc. J\'anos Bolyai,
Vol. 10, North--Holland, Amsterdam, 1975.    

\end{thebibliography}
\end{document}